\documentclass[a4paper,11pt]{article}

\usepackage[toc,page]{appendix}

\usepackage{proof}

\usepackage{url}

\usepackage{amsthm}
\usepackage{mathtools}
\usepackage{latexsym}
\usepackage{amssymb}
\usepackage{verbatim}
\usepackage{setspace}
\usepackage{qtree}
\usepackage{enumitem}
\usepackage{xcolor}
\usepackage{url}
\usepackage{stmaryrd}

\usepackage{authblk}

\usepackage[margin=1.4in]{geometry}

\usepackage{tikz}
\usetikzlibrary{arrows,decorations.pathmorphing,backgrounds,positioning,fit}

\newcommand{\fo}{\mathrm{FO}}

\newcommand{\lfo}{\ensuremath{\mathrm{BndSCL}}}
\newcommand{\ulfo}{\ensuremath{\mathrm{SCL}}}

\newcommand{\substshift}{\ensuremath{L\text{Subst}1}}
\newcommand{\substcopyelim}{\ensuremath{L\text{Subst}2}}

\DeclareMathOperator{\subf}{Sf}
\DeclareMathOperator{\rfor}{Rf}

\DeclareMathOperator{\ran}{\mathrm{ran}}

\DeclareMathOperator{\uso}{\forall \mathrm{SO}}
\DeclareMathOperator{\eso}{\exists \mathrm{SO}}

\newbox\gnBoxA
\newdimen\gnCornerHgt
\setbox\gnBoxA=\hbox{$\ulcorner$}
\global\gnCornerHgt=\ht\gnBoxA
\newdimen\gnArgHgt
\def\Godelnum #1{%
\setbox\gnBoxA=\hbox{$#1$}%
\gnArgHgt=\ht\gnBoxA%
\ifnum     \gnArgHgt<\gnCornerHgt \gnArgHgt=0pt%
\else \advance \gnArgHgt by -\gnCornerHgt%
\fi \raise\gnArgHgt\hbox{$\ulcorner$} \box\gnBoxA %
\raise\gnArgHgt\hbox{$\urcorner$}}

\author{Reijo Jaakkola}
\author{Antti Kuusisto}
\affil{Tampere University and University of
Helsinki, Finland}

\date{}

\begin{document}

\setlength\abovedisplayskip{3pt}
\setlength\belowdisplayskip{3pt}
\title{First-order logic with self-reference}

\theoremstyle{plain}
\newtheorem{theorem}{Theorem}[section]
\newtheorem{lemma}[theorem]{Lemma}
\newtheorem{corollary}[theorem]{Corollary}
\newtheorem{proposition}[theorem]{Proposition}
\newtheorem{fact}[theorem]{Fact}
\theoremstyle{definition}
\newtheorem{definition}[theorem]{Definition}
\newtheorem{remark}[theorem]{Remark}
\newtheorem{example}[theorem]{Example}

\maketitle

\begin{abstract}
\noindent
   We consider an extension of first-order logic with a recursion operator that corresponds to allowing formulas to refer to themselves. We investigate the obtained language under two different systems of semantics, thereby obtaining two closely related but different logics. We provide a natural deduction system that is complete for validities for both of these logics, and we also investigate a range of related basic decision problems. For example, the validity problems of the two-variable fragments of the logics are shown coNexpTime-complete, which is in stark contrast with the high undecidability of two-variable logic extended with least fixed points. We also argue for the naturalness and benefits of the investigated approach to recursion and self-reference by, for example, relating the new logics to Lindstr\"{o}m's Second Theorem.
\end{abstract}


\section{Introduction}

This paper investigates an extension of 
first-order logic $\mathrm{FO}$ with an 
operator that allows formulas to refer to themselves. 
The idea is simple. We extend the syntax of $\mathrm{FO}$ by 
the following two rules:

\begin{enumerate}
    \item 
    If $\varphi$ is a 
    formula, then so is $L\varphi$. Here $L$ is 
    a \emph{label symbol} intuitively naming 
    the formula $\varphi$.
    \item
    If $L$ is a label symbol, then $C_L$ is an
    atomic formula (a \emph{claim symbol}) 
    intuitively referring to 
    the formula labelled by $L$. 
\end{enumerate}

We interpret formulas via extending the 
standard game-theoretic semantics for $\mathrm{FO}$ by
the rule that if an atom $C_L$ is reached in the 
play of the semantic game, then the players jump 
back to the formula $L\varphi$ and the game 
continues from there. Thereby the symbol $L$
can indeed be seen as a \emph{naming} or \emph{labelling} operator 
that names $\varphi$, while $C_L$ is a \emph{claming}
operator claiming that $\varphi$ holds.  
Other rules are precisely as in
standard first-order logic, making our system a 
conservative extension of $\mathrm{FO}$.

We give two alternative semantics to 
the obtained language, called \emph{bounded} 
and \emph{unbounded} semantics. In the unbounded 
semantics, the players continue until (if ever) an
atomic first-order formula is reached, and the 
play is then won according to the same criteria as in
$\mathrm{FO}$. If the game play continues forever,
neither player wins. 
The bounded semantics is similar, but
there the players must commit to an integer value giving
the number of times formulas can be repeated, i.e., how
many jumps from claim symbols $C_L$ back to 
label symbols $L$ are allowed. This forces all
plays to be of finite duration. The winner is
decided in the same way as in
unbounded semantics and can occur only if an $\mathrm{FO}$-atom is
reached. If the players simply 
run out of time, then neither player wins the play.

Under the unbounded semantics, our logic is a 
fragment of the logic $\mathrm{CL}$, or 
\emph{computation logic}, introduced in \cite{turingcomp}
and discussed further in, e.g., \cite{boundedmucalctwo}, 
\cite{partial}.
In addition to the looping operator 
studied in this paper, $\mathrm{CL}$
extends $\mathrm{FO}$ with the 
capacity to modify models by adding and deleting
domain points as well as tuples of relations. Yet $\mathrm{CL}$ is a conservative extension of $\mathrm{FO}$,
giving the same interpretations to all
first-order operators as $\mathrm{FO}$ via 
game-theoretic semantics. 
$\mathrm{CL}$ captures the class $\mathrm{RE}$
in the sense of
descriptive complexity theory, that is, $\mathrm{CL}$
can define precisely the classes of finite models 
that are recursively enumerable \cite{turingcomp}. 
In fact, more is true. We can associate
Turing machines $\mathrm{TM}$ with formulas $\varphi_{
\mathrm{TM}}$ so that 
\begin{enumerate}
    \item
    $\mathfrak{M}\models\varphi_{\mathrm{TM}}$ iff
    $\mathrm{TM}$ accepts the encoding of $\mathfrak{M}$,
    \item
    $\mathfrak{M}\models\neg\varphi_{\mathrm{TM}}$ iff
    $\mathrm{TM}$ rejects the encoding of $\mathfrak{M}$,
    \item
    $\varphi_{TM}$ is indeterminate on $\mathfrak{M}$
    iff $\mathrm{TM}$ diverges on the encoding of $\mathfrak{M}$.
\end{enumerate}
Note here that the game-theoretic negation is 
strongly constructive: while $\mathfrak{M}\models\varphi$
means that the proponent has a winning strategy in the
game for $\mathfrak{M}$ and $\varphi$, 
negation is 
defined such that $\mathfrak{M}\models\neg\varphi$ if
the opponent has a winning strategy in the game.
While $\mathrm{FO}$ is determined, $\mathrm{CL}$ has 
formulas that are not, so neither player has a winning
strategy. This is necessary for capturing the full 
expressive power of Turing machines in the
way $\mathrm{CL}$ does, creating an exact match also 
between indeterminacy of formulas and diverging computations.
Verifiability of a formula in a model is of
course matched with 
acceptance and falsifiability (in a model) with rejection, i.e., 
halting in a rejecting state. 
Thus there is a full symmetry 
between logic and computation, 
game-theoretic semantics being 
the key for achieving this. In this 
context, $\neg$ is most naturally read to
indicate falsifiability, although on
first-order formulas, the involved mode of
falsifiability collapses to classical
negation. Note also that the self-reference mechanisms have immediate readings in natural language, so the framework produces formulas that have simple natural language counterparts. See \cite{games19} for further discussions on this.

We call the logics studied in the current
paper $\mathrm{SCL}$
and $\mathrm{BndSCL}$, for \emph{static 
computation logic} and \emph{bounded static computation
logic}. The logic $\ulfo$ follows the 
unbounded version of our semantics and is thus a
fragment of $\mathrm{CL}$ also
semantically, because $\mathrm{CL}$ is
defined in \cite{turingcomp} based on the unbounded semantics.
The term \emph{static} here refers to the 
fact that $\ulfo$ and $\lfo$ do not
modify the models under
investigation, unlike $\mathrm{CL}$. We
note that quite 
naturally, we could 
alternatively refer to
$\ulfo$ and $\lfo$, e.g., as
\emph{non-well-founded} $\mathrm{FO}$ under
bounded and unbounded semantics.

The logics $\ulfo$ and $\lfo$ also relate to 
several other formalisms studied in the literature.
The looping mechnism is similar to that of
the modal $\mu$-calculus \cite{bradfield}, which
becomes especially apparent when considering the 
game-theoretic approach to its semantics via
parity games. The parity condition essentially 
allows for the $\mu$-calculus to be
closed under
classical negation. In contrast, the
logics $\ulfo$ and $\lfo$---being based on a
strongly constructive game-theoretic 
negation---allow for indeterminate 
formulas but are based on simple reachability
games that can be won only by
ending up in an $\mathrm{FO}$-atom---as in first-order logic.
Nevertheless, fixed-point logics in 
general bear similarities to $\ulfo$ and $\lfo$.

Perhaps the best known fixed-point logic is
$\mathrm{LFP}$, or \emph{least fixed-point 
logic}, see \cite{moschovakis} for an 
early thorough approach to the formalism.
In $\mathrm{LFP}$, the use of negation is limited to
guarantee monotonicity of the iterated operators. 
As negation can be used 
entirely freely in $\ulfo$ and $\lfo$, they are syntactically perhaps more closely related to 
\emph{partial fixed-point logic} $\mathrm{PFP}$ than  $\mathrm{LFP}$.
In $\mathrm{PFP}$, a non-converging
computation is interpreted as $\bot$, thereby
essentially forcing the involved procedure to converge (see, e.g., \cite{ebbingfmt} for an introduction to $\mathrm{PFP}$).
In contrast, $\ulfo$ and $\lfo$ allow for diverging
formulas and the game-theoretic and coinductive 
approach to their semantics is not
based on fixed points in any direct way.
Much of the naturality of the setting stems from the 
strongly constructive negation, making 
verification of a negated
formula in a model $\mathfrak{M}$
equivalent to falsifying the formula in $\mathfrak{M}$.
This feature is present even in $\mathrm{CL}$, and
also contrasts with intuitionistic logic, as double 
negation cancels in $\mathrm{CL}$ (and thus also in $\ulfo$ and $\lfo$).
Having said all this, it
nevertheless ought to be kept in 
mind that both $\ulfo$ and $\lfo$ are conservative 
extensions of $\mathrm{FO}$ and thus negation 
indeed behaves entirely classically when
restricting to $\mathrm{FO}$-formulas.

The difference between $\lfo$ and $\ulfo$ is 
that in $\lfo$, the players must
commit to a maximum number of
times the self-referential formulas can be repeated when playing the semantic game. 
This idea, which is conceptually related to the difference between 
\emph{for-loops} and \emph{while-loops}, has
been investigated in
different forms in various different studies. The papers 
\cite{atl1c}, \cite{acmtrans2018}, \cite{atl2}, \cite{atl2c}
utilize bounded semantics in \emph{alternating-time
temporal logic} $\mathrm{ATL}$ and its
variants. The results 
concern, e.g.,
identifying a hierarchy of fixed-parameter 
tractable variants of the
extension $\mathrm{ATL}^+$ of $\mathrm{ATL}$.
The papers \cite{time2017} and \cite{tcs2019} 
develop deduction systems and tableaux for $\mathrm{ATL}$ 
under bounded semantics, and the 
articles \cite{boundedmucalcone}, \cite{boundedmucalctwo}
device a bounded game-theoretic semantics for the
modal $\mu$-calculus and show it equivalent to the
standard one. It is also shown
that in the new setting, semantic games of
the $\mu$-calculus
always end after a finite number of rounds, even in
infinite models. The $\mu$-calculus formula size
games of \cite{vilander} are based on
this bounded semantics in an 
essential way. Also concerning bounded semantics, \cite{boundedmucalctwo}
studies the modal fragment of $\mathrm{CL}$, called $\mathrm{MCL}$, and
observes that while it has $\textsc{PTime}$-model
checking and a
nice bounded semantics with short game durations, it can nevertheless
easily express $\textsc{PTime}$-complete
properties such as 
alternating reachability.
Concerning yet further relevant works, we
stress that there are numerous logics that relate to self-reference and 
recursion, too numerous to
detail here. Relating to axiomatizations, the seminal work \cite{walukiewicz} on
the $\mu$-calclulus should be mentioned. Concerning self-reference,
\cite{bolander} gives a general overview on the topic. It is also worth noting that various
directions in infinitary
logic, especially infinitely deep formulas, bear 
technical links to our work.

\subsection{Contributions}

One of the main aims of this paper is to
provide a complete proof system for $\ulfo$
and $\lfo$. Interestingly, it turns out that
both of these logics have the same set of 
validities. Below we provide a natural deduction
system that is
complete for validities of the logics. We also  
show that if $\Sigma$ is a set of first-order formulas, then
$\Sigma\models \varphi$ iff $\Sigma\vdash\varphi$ holds
for both logics.

Furthermore, we investigate the expressive powers and 
computational properties of $\ulfo$ and $\lfo$ and
their fragments. 
We identify several interesting
properties that are straightforward to 
express in $\ulfo$ or $\lfo$ while not 
being expressible in $\mathrm{FO}$. For example, it is
easy to express in $\ulfo$ that a linear order is
well-founded, whence it is easy to define the intended 
model of arithmetic up to isomorphism in that logic.
Concerning computational properties, perhaps 
most notably, we show that the two-variable
fragments of $\ulfo$ and $\lfo$ have
\textsc{coNexpTime}-complete validity problems.
This is in stark contrast with validity for two-variable 
logic with fixed points, which is
highly undecidable, having been shown $\Pi_1^1$-hard in \cite{logicswithtwovariables}. This nicely demonstrates the
possibilities of using recursion in the
way used in this article. In addition to positive 
results, we also show, for example, that the 
satisfiability problem of $\lfo$ is $\Sigma_2^0$-complete.

To better understand the features of $\ulfo$ and $\lfo$, we investigate their model theory. First, we establish that both of these logics have the countable downwards Löwenheim-Skolem property: if $\mathfrak{A}$ is a model of $\varphi$, then $\mathfrak{A}$ has a countable substructure which is also a model of $\varphi$. Secondly, we show that neither of these logics enjoys the Craig interpolation property. Finally, we investigate determinacy of sentences, i.e., the question whether a sentence $\varphi$ has the property that in every model, one of the players has a winning strategy in the semantic
game for $\varphi$. Note that this is
equivalent to asking whether $\varphi\vee\neg\varphi$ is valid. Interestingly, it turns out that $\varphi\vee\neg\varphi$ is valid precisely when $\varphi$ is equivalent to a first-order sentence. We also give an example demonstrating that the above correspondence fails if we restrict our attention to finite models.

We also investigate $\Pi_1^1$-relations as 
well as $\Sigma_{\omega + 1}^0$-relations.
A well-known theorem of Kleene states that over the standard structure of natural numbers $\mathbb{N}$, the class of inductive relations and $\Pi_1^1$-relations (or relations definable in universal second-order logic over $\mathbb{N}$) coincide \cite{Kleene1955OnTF}. In \cite{Harel1984APL}, an alternative characterisation of $\Pi_1^1$-relations was given in terms of programs in the programming language IND. Inspired by these characterisations, we give yet another
characterisation of $\Pi_1^1$-relations by
showing that they also coincide with the class of $\ulfo$-definable relations. We also study the class of $\lfo$-definable relations and prove that they coincide with the class of $\Sigma_{\omega + 1}^0$-relations. To the best of our knowledge, this is the first logical characterisation of $\Sigma_{\omega + 1}^0$-relations. Furthermore, it
sheds light on the expressive power of $\lfo$.

As a final remark, we describe one of
the most important results on $\ulfo$ and
$\lfo$ we have obtained.
Firstly, these logics 
are way more 
expressive than $\mathrm{FO}$.
Secondly, they nevertheless have
recursively enumerable sets of validities,
and, as discussed above, the downward L\"{o}wenheim-Skolem
property. This contrasts with Lindstr\"{o}m's
Second Theorem, which states the the
expressive power of an effectively
regular logic with
recursively enumerable validities and the downward L\"{o}wenheim-Skolem property should not exceed that of $\mathrm{FO}$ (see, e.g., \cite{ebbingmathlog}).
The nice thing is that the only
property that $\ulfo$ and $\lfo$ lack in
being effectively regular is closure under
classical negation, and,
they nevertheless are both
closed under the highly
natural strong negation.\footnote{We note that in this article, we prove this statement explicitly only for the variant of $\mathrm{FO}$ without constant and function symbols. This is to keep the work simple. However, it is trivial to extend our study to involve constants and function symbols.} Indeed, in the case
of $\mathrm{CL}$, the logic \emph{cannot} be closed under classical negation, as
$\mathrm{CL}$ captures $\mathrm{RE}$. Furthermore, we stress once more that the game-theoretic negation is simply the
plain classical negation
when limiting to the first-order fragment.

\section{Preliminaries}

We denote the natural numbers by $\mathbb{N}$,
the integers by $\mathbb{Z}$ and
the positive integers by $\mathbb{N}_+$.
A \textbf{linear order structure} $(A,<^A)$ is a
structure where $<^A$ is a strict linear order over
the domain set $A$.
 A \textbf{discrete order structure} $(A,<^A)$ is a structure where $<^A$ is a strict linear order over the domain $A$ such that the following conditions hold.
\begin{enumerate}
    \item
    The linear order has a minimun element $0^A\in A$.
    \item
    Each element $a\in A$ has a unique
    successor element $b\in A$ in the case $a$ has a
    successor at all. That is, if
    there is some $d\in A$
    such that $a<^A d$, then 
    there exists an element $b\in A$ such
    that $a<^A b$ and for all $c\in A\setminus\{a,b\}$, we
    have $c <^A a$ or $b <^A c$. 
\end{enumerate}

A \textbf{finite sequence} in $V$ is a finite
tuple $(v_1,\dots , v_n)$ of
elements $v_i\in V$. The element $v_n$ is 
the \textbf{last element} of the tuple.
An \textbf{$\omega$-sequence} in $V$ is an
infinite tuple $(v_i)_{i \in \mathbb{N}_+} = 
(v_1,v_2\dots )$ of elements $v_i\in V$. 
Here $\omega$ denotes the first 
infinite ordinal. The 
element $v_1$ is
the \textbf{first element} of 
both $(v_1,\dots , v_n)$ and $(v_1,v_2\dots )$.
If $p = (v_1,\dots , v_n)$
and $q = (u_1,\dots , u_m)$ are finite sequences, 
their \textbf{concatenation} $(v_1,\dots ,v_n,
u_1,\dots , u_m)$ is denoted by $p \cdot q$.
A singleton sequence $(v)$ is identified with $v$.
We sometimes denote tuples with vector 
notation, e.g., $\overline{v}$ denotes a 
tuple of elements $v_i$.

In this paper, a \textbf{directed graph} $(V,E)$ is a
structure where $V$ is any (possibly infinite) set
and $E\subseteq V \times V$. Thus directed 
graphs are allowed to have \textbf{reflexive loops}, i.e.,
the set $E$ may contain
pairs $(v,v)$. A \textbf{dead end} in $(V,E)$ is an 
element $v\in V$ such that there does not exist
exists any $u$ such that $(v,u) \in E$.
A \textbf{walk} in a 
directed graph $(V,E)$ is either a
finite or an $\omega$-sequence in $V$
such that we have $(v_i,v_{i+1}) \in E$
for each pair $(v_i, v_{i+1}) \in E$ of
subsequent elements in the sequence. We note
that in the literature, walks are
often defined as sequences of 
edges, but our definition is more convenient for
this paper. A walk is a \textbf{path} if it
does not repeat any element.
For $v\in V$, the set of finite (nonempty)
walks with the first element $v$ is denoted by $V^*_{walk}(v)$.

A \textbf{game arena} is a tuple $(V_0,V_1,E)$ 
where $V_0$ and $V_1$ are any disjoint sets and
$E \subseteq V\times V$ for $V := V_0 \cup V_1$.
Intuitively, the arena is a platform for a two-player 
game where $V_0$ is a set of positions 
for \textbf{player 0} and $V_1$ for \textbf{player 1}. In
each position $v\in V_0$ (respectively, $v\in V_1$), 
player $0$ (respectively, player 1) chooses a node $u$
such that $(v,u)\in E$ and the
players then continue from the
new position $u$.
We define here that a \textbf{play} on the arena is a
maximal walk in $(V,E)$, where maximality means 
that the walk is either infinite (of 
length $\omega$) or finite with its last 
element being a dead end. A \textbf{generalized winning condition}
over the arena $(V_0,V_1,E) = (V,E)$ is a pair
$(S_0,S_1)$ where $S_0$ and $S_1$ are sets of plays.
The set $S_0$ (respectively, $S_1$) lists the
plays that player 0 (player 1) wins.
Note that it is possible that neither of
the players---or even both players---win a play.

A \textbf{game} is a triple that 
specifies a game arena, a \textbf{beginning position} 
(which is a node $v\in V$)
and a generalized winning condition $(S_0,S_1)$ consisting of 
plays with the first position $v$. In a game with first position $v$, a \textbf{strategy} of player 0 (respectively, player 1) is a
function $f:U \rightarrow V$, 
where $U$ is the set of finite walks with the last element in $V_0$ (respectively, in $V_1$) and the
first element $v$. A strategy $f$ is
\textbf{followed} in a play $p$ if every prefix $q$ of $p$
with $q \in \mathit{dom}(f)$ has the property
that $f(q \cdot f(q))$ is, likewise, a prefix of $p$. We may also talk about following a strategy in a prefix of a play; the meaning of this is defined in the obvious way.
A strategy of player 0 (respectively, player 1) is a \textbf{winning strategy} if every play where $f$ is followed
belongs to $S_0$ (respectively, $S_1$).
Letting $S$ denote the set of all
possible plays of a game, a strategy of player 0 (respectively, player 1) is a \textbf{non-losing strategy} if every play 
where $f$ is followed
belongs to $S\setminus S_1$ (respectively, $S\setminus S_0$).
A strategy $f$ is \textbf{positional} if it depends only on
the last position of its
inputs, i.e., $f(q \cdot v) = f(q' \cdot v)$ for all 
prefixes of plays $q \cdot v$ and $q' \cdot v$ in $\mathit{dom}(f)$.
Note that we can identify such a
strategy $f$ of player $i\in \{0,1\}$ by the function $g : V_i \rightarrow V$ such that $g(v) = f(q\cdot v)$
for all $(q \cdot v) \in \mathit{dom}(f)$.
A game is \textbf{determined} if precisely one player has a winning strategy in it. It is \textbf{positionally determined} if precisely one
player has a positional winning strategy in it. Note that positional determinacy implies determinacy.

A \textbf{reachability game} for player $i\in \{0,1\}$ is a game where the winning condition $S_i$ of player $i$ contains precisely the plays $(v_1,\dots , v_n)$ where $v_n$ is a 
dead end belonging to the opponent, i.e., $v_n\in V_{j}$ for $j\in \{0,1\}\setminus \{i\}$. The complement of $S_i$ defines a \textbf{safety game} for the opponent $j$, that is, if the set of plays $S_i$ defines a reachability game for player $i$, then the
complement set of plays defines a safety game for the opponent of $i$. The complement set $S_j$ contains precisely those finite plays that end in a dead end $v_n\in V_i$ for $i$ and all
infinite plays.
Reachability games (and thus safety games) are sometimes 
represented by structures $(V,E,v_1,V_0,V_1)$ (or by
close variants of this representation) in a natural way such that $V = V_0\cup V_1$, $E\subseteq V\times V$ and $v_1$ is the beginning position. In this representation, there is no need to encode all the
plays leading to a win of player $i$, as
obviously only the final position of each finite play matters.
The following result is well known and follows directly from, e.g., \cite{gamesgr}. It states that in any reachability or safety game,
precisely one of the players has a winning strategy, and 
that strategy can be assumed positional.

\begin{theorem}\label{reachabilitygametheorem}
    Reachability and safety games are positionally determined (even on infinite arenas).
\end{theorem}

Leaving games behind for now, we denote models by $\mathfrak{A},\mathfrak{B}$,
and so on. The domain of a model is denoted by the corresponding Roman capital letter, so
for example $A$
denotes the domain of $\mathfrak{A}$.
An \textbf{assignment} for a model $\mathfrak{A}$
is a function $s:V \rightarrow A$ where $V$ is
some (often finite) set of variable symbols. Note that also $\varnothing$ is an assignment (for empty $V$). An assignment 
mapping into a set $A$ is called an $A$-assignment. An assingment that is otherwise as $s$ but 
sends $x$ to $a$ is denoted by $s[a/x]$.
A model $\mathfrak{A}$ (respectively, and assignment $s$) is \textbf{$\varphi$-suitable} if the
vocabulary of $\mathfrak{A}$ contains the
vocabulary of $\varphi$ (respectively, 
the domain of $s$ contains the 
free variables of $\varphi$). When the
specification of $\varphi$ is
sufficiently clear for and 
from the context under investigation, we may simply call $\mathfrak{A}$
and $s$ \textbf{suitable}.

In this article, the language of \textbf{first-order logic} $\mathrm{FO}$ includes
equality and $\bot$ as primitives and 
contains the Boolean operators $\neg$, $\wedge$, $\vee$ and
the quantifiers $\exists$ and $\forall$. We may use $\top$, $\rightarrow$ and $\leftrightarrow$ as abbreviations in the
usual way. We 
limit to purely relational
vocabularies for the sake of
simplicity and brevity.\footnote{Indeed, this limitation could 
easily be lifted.} Atomic formulas
belonging to $\mathrm{FO}$ are
called \textbf{$\mathrm{FO}$-atoms}, or \textbf{first-order atoms}. This is to distinguish them from \emph{claim symbols}, to be formally introduced later on. 
\textbf{Universal
second-order logic} $\forall\mathrm{SO}$ is the
fragment of second-order logic with formulas of the
form $\forall X_1\dots \forall X_n\, \psi$ where 
$X_1,\dots , X_n$ are second-order relation
variables and $\psi$ is a formula of $\mathrm{FO}$.

\begin{definition}
    Let $\varphi$ be a formula of $\fo$, $\mathfrak{A}$ a suitable model and $r$ a suitable assignment. We define the \textbf{evaluation game} $\mathcal{G}(\mathfrak{A},r,\varphi)$ as follows. The game has two players, \textbf{Abelard} and \textbf{Eloise}. The positions of the game are tuples $(\psi,s,\#)$, where $s$ is a $\psi$-suitable $A$-assignment and $\# \in \{+,-\}$.
    The game begins from the initial position $(\varphi,r,+)$ and it is then played according to the following rules.
    
    \begin{itemize}
\item In a position $(\alpha, s, +)$, where $\alpha$ is an $\fo$-atom, the play of the game ends and Eloise wins if $\mathfrak{A},s\models \alpha$. Otherwise Abelard wins.
\item In a position $(\alpha, s, -)$, where $\alpha$ is an $\fo$-atom, the play of the game ends and Abelard wins if $\mathfrak{A},s\models \alpha$. Otherwise Eloise wins.
%
%
%
%
%
%
        \item In a position $(\neg \psi, s, +)$, the game continues from the position $(\psi,s,-)$. Symmetrically, in a position $(\neg \psi, s, -)$, the game continues from the position $(\psi,s,+)$.
        \item In a position $(\psi \land \theta, s, +)$, Abelard chooses whether the game continues from the position $(\psi,s,+)$ or $(\theta,s,+)$.
        \item In a position $(\psi \wedge \theta, s, -)$, Eloise chooses whether the game continues from the position $(\psi,s,-)$ or $(\theta,s,-)$.
                \item In a position $(\psi \vee \theta, s, +)$, Eloise chooses whether the game continues from the position $(\psi,s,+)$ or $(\theta,s,+)$.
        \item In a position $(\psi \vee \theta, s, -)$, Abelard chooses whether the game continues from the position $(\psi,s,-)$ or $(\theta,s,-)$.
                \item In a position $(\forall x \psi, s, +)$, Abelard chooses some element $a\in A$ and the game continues from the position $(\psi, s[a/x],+)$.
        \item In a position $(\forall x \psi, s, -)$, Eloise chooses some element $a\in A$ and the game continues from the position $(\psi, s[a/x], -)$.
        \item In a position $(\exists x \psi, s, +)$, Eloise chooses some element $a\in A$ and the game continues from the position $(\psi, s[a/x],+)$.
        \item In a position $(\exists x \psi, s, -)$, Abelard chooses some element $a\in A$ and the game continues from the position $(\psi, s[a/x], -)$.
        %
        %
        %
        %
        %
        %
        %
         %
         %
         %
    \end{itemize}
    
If $r$ is the empty assignment $\varnothing$ (and hence $\varphi$ is a sentence), we may write $\mathcal{G}(\mathfrak{A},\varphi)$ instead of $\mathcal{G}(\mathfrak{A},\varnothing,\varphi)$.
\end{definition}

\begin{definition}
    Let $\varphi$ be an $\fo$-formula, $\mathfrak{A}$ a suitable model and $s$ a suitable assignment. We define that $\varphi$ is \textbf{true} (or \textbf{verifiable}) in $\mathfrak{A}$ under $s$, denoted by $\mathfrak{A},s \models \varphi$,
    iff Eloise has a winning strategy in the game $\mathcal{G}(\mathfrak{A},s,\varphi)$. If $\varphi$ is a sentence, we may write $\mathfrak{A}\models \varphi$ if $\mathfrak{A},\varnothing\models\varphi$, where $\varnothing$ is the empty assignment. We then simply say that $\varphi$ is true (or verifiable) in $\mathfrak{A}$.
\end{definition}

The above specifies the standard \textbf{game-theoretic semantics} for $\mathrm{FO}$.
It is well known and easy to see that the above definition via evalulation games agrees with the standard Tarski semantics for $\mathrm{FO}$, i.e., $\varphi$ is true in $\mathfrak{A}$
under $s$ according to the game-theoretic semantics iff the same holds in the sense of Tarski semantics.


We then extend the syntax of $\mathrm{FO}$.
Define the set $\mathit{LBS} := \{ L_i\, |\, i\in\mathbb{N}\}$ of
\textbf{label symbols}, and based on this, define the set $\mathit{RFS} := 
\{ C_{L_i}\, |\, L_i\in \mathit{LBS}\}$ of
\textbf{reference symbols}, also called
\textbf{claim symbols}.
The syntax of the logics $\ulfo$
and $\lfo$ is obtained by extending the
formula construction rules of $\fo$ by 
the following rules.
\begin{itemize}
\item
Each claim symbol $C_L\in \mathit{RFS}$ is an atomic formula.
    \item 
If $\varphi$ is a
formula and $L\in \mathit{LBS}$,
then $L\, \varphi$ is a formula.
\end{itemize}

Reference symbols can also be called \textbf{non-$\mathrm{FO}$-atoms} or \textbf{looping atoms},
while the remaining atomic formulas are 
\textbf{$\mathrm{FO}$-atoms}. Now,
consider a formula $\varphi$ and an
occurrence $C_L$ of a
reference symbol in $\varphi$. The \textbf{reference formula} of $C_L$, denoted $\rfor(C_L)$,
 is the subformula occurrence $L \psi$ of $\varphi$ such
 that there is a directed path 
 from $L\psi$ to $C_L$ in the syntax tree of $\varphi$, and 
 $L$ does not occur strictly between $L\psi$ and $C_L$ on
 that path. Note that 
the reference formula of $C_L$ is unique if it
exists at all. When we 
talk about reference 
formulas, we mean
reference formula occurrences. A claim
symbol occurrence $C_L$ is in the \textbf{strict
scope} of a label symbol occurrence $L$
when $C_L$ can be reached in the 
syntax tree from the node
with $L$ via a
directed path that does not contain 
further occurrences of $L$. For example, in
$LLC_L$, the atom $C_L$ is in the strict
scope of the second but not the first 
occurrence of $L$. However, the atom is in the scope of both occurrences. A looping atom 
occurrence $C_L$ is $\textbf{free}$ in a
formula if it is not in the 
scope (equivalently, strict scope) of any occurrence of $L$. A 
label symbol occurrence $L$ is \textbf{dummy} if
there are no corresponding looping atoms $C_L$ in
the strict scope of $L$. A formula $\varphi$ is 
\textbf{regular} if the following conditions hold.
\begin{enumerate}
\item
No label symbol $L$ occurs 
more than once in it.
\item
If an atom $C_L$ occurs free in
the formula, then the corresponding
label symbol $L$ does not occur 
anywhere in the formula.
\end{enumerate}

The set of \textbf{subformulas} of a formula $\varphi$ is
denoted by $\subf(\varphi)$. As usual in game-theoretic semantics, subformulas mean
subformula occurrences, so for example $P(x)\vee P(x)$ has three subformulas, the disjunction itself and the left and the right occurrences of $P(x)$.
The set $\subf_L(\varphi)$ of
\textbf{$L$-subformulas} (or 
\textbf{label subformulas}) of $\varphi$ is
the set of subformulas of
type $L'\psi$ (where $L'$ is any 
label symbol) in $\varphi$. This includes $\varphi$ itself if $\varphi$ is of type $L'\varphi'$.
We say that $\varphi$ is in \textbf{weak negation normal form} if the only negated subformulas of $\varphi$ are atomic formulas. We say that $\varphi$ is in \textbf{strong negation normal form} if the only negated subformulas of $\varphi$ are atomic $\fo$-formulas.

The set of free variables of a
formula $\varphi$ of $\ulfo$ of $\lfo$ is
defined inductively in the same way as for $\mathrm{FO}$-formulas, with 
the following two additional rules.

\begin{itemize}
\item
The set of free variables of a claim symbol $C$ is $\varnothing$.
    \item 
    The free variables of $L\psi$ is the
    the same as that of $\psi$.
\end{itemize}

Analogously to the case for $\mathrm{FO}$, a
model $\mathfrak{A}$ and assignment $s$ are 
called suitable for $\varphi$ if $s$ interprets
the free variables of $\varphi$ in $A$ 
while $\mathfrak{A}$ interprets the relation 
symbols in $\varphi$.

\begin{definition}
The semantics of $\ulfo$ is given via a 
game that extends the game for $\mathrm{FO}$ by
the following rules.

\begin{itemize}
\item In a position $(C_L, s, \#)$, 
where $\#\in\{-,+\}$, the game continues
from the position $(L\psi, s, \#)$ where $L\psi$ is
the reference formula of $C_L$. In the case there exists no
such reference formula, the play of the game ends and
neither of the players win the play.
\item
In a position $(L \psi, s, \#)$, where $\#\in\{+,-\}$, the game simply continues from the position $(\psi, s, \#)$.
\end{itemize}
The game for $\mathfrak{A}$, $\varphi$
and $s$ is denoted by $\mathcal{G}_\infty(
\mathfrak{A},s,\varphi)$. If $s$ is the empty assignment $\varnothing$ (and hence $\varphi$ is a sentence), we may write $\mathcal{G}_\infty(\mathfrak{A},\varphi)$ instead of $\mathcal{G}_\infty(\mathfrak{A},\varnothing,\varphi)$.
\end{definition}

Note that infinite plays are won by neither player. Winning occurs only if an $\mathrm{FO}$-atom is reached, exactly as in first-order logic.
The game  $\mathcal{G}_\infty(\mathfrak{A},s,\varphi)$ is
clearly a reachability game for Eloise, and thereby, by Theorem \ref{reachabilitygametheorem}, Eloise has a positional winning strategy if and only if she has a general one. We note that this holds despite the fact that the
underlying models are not required to be finite. Furthermore, the same claims hold for Abelard as well.

The semantics of $\ulfo$ is formally defined as follows.

\begin{definition}
    Let $\varphi$ be a formula of $\ulfo$, $\mathfrak{A}$ a suitable model and $s$ a suitable assignment. We define that $\varphi$ is \textbf{true} (or \textbf{verifiable}) in $\mathfrak{A}$ under $s$, denoted $\mathfrak{A},s \models \varphi$,
    iff Eloise has a winning strategy in the game $\mathcal{G}_\infty(\mathfrak{A},s,\varphi)$. If $\varphi$ is a sentence, we may write $\mathfrak{A}\models \varphi$ if $\mathfrak{A},\varnothing\models\varphi$. We then say that $\varphi$ is true (or verifiable) in $\mathfrak{A}$.
\end{definition}

We also define two notions of
equivalence for $\ulfo$

\begin{definition}\label{ulfoequivalencedefinition} 
Let $\varphi$ and $\psi$ be a formulas of $\ulfo$.
The formulas are \textbf{weakly
equivalent} if the equivalence
$$\mathfrak{A},s\models \varphi\ \Leftrightarrow\ 
\mathfrak{A},s\models\psi$$
holds for all $\mathfrak{A}$ and $s$ that 
are suitable with respect to both $\varphi$ and $\psi$.
The formulas $\varphi$ and $\psi$ are \textbf{strongly equivalent} if they are weakly equivalent and 
the equivalence

\begin{align*}
    &\text{Abelard has a 
winning 
strategy in }\mathcal{G}_{\infty}(\mathfrak{A},s,\varphi)\\
\Leftrightarrow & \\
&\text{ Abelard has a winning
strategy in }\mathcal{G}_{\infty}(\mathfrak{A},s,\psi)
\end{align*}

\medskip

\medskip

\medskip

\noindent
holds for all $\mathfrak{A}$ and $s$ that 
are suitable with respect to $\varphi$ and $\psi$.
\end{definition}

It is easy to see that an alternative
way to formulate strong
equivalence of $\varphi$ and $\psi$ is to require
that for any suitable $\mathfrak{A}$ and $s$, 
precisely one of the following three conditions hold.

\begin{enumerate}
    \item 
    Eloise has a winning 
    strategy in both games $\mathcal{G}_{\infty}(\mathfrak{A},s,\varphi)$
    and $\mathcal{G}_{\infty}(\mathfrak{A},s,\psi)$.
    \item
    Abelard has a winning 
    strategy in both  $\mathcal{G}_{\infty}(\mathfrak{A},s,\varphi)$
    and $\mathcal{G}_{\infty}(\mathfrak{A},s,\psi)$.
    \item
    Neither of the players has a winning
    strategy in $\mathcal{G}_{\infty}(\mathfrak{A},s,\varphi)$,
    and the same holds also for the game
    $\mathcal{G}_{\infty}(\mathfrak{A},s,\psi)$.
\end{enumerate}

To define the semantics of $\lfo$, we next
define two related games, the second one
extending the first one. The intuitive idea is simply that in the beginning of each play, the players commit to some maximum duration of the play.

\begin{definition}
    Let $\varphi$ be a formula of $\lfo$, and suppose $\mathfrak{A}$ is a suitable model and $r$ a suitable assignment. Let $n \in \mathbb{N}$. We define the \textbf{$n$-bounded evaluation game} $\mathcal{G}_n(\mathfrak{A},r,\varphi)$ as follows. The game has two players, Abelard and Eloise. The positions of the game are tuples $(\psi,s,\#,m)$ where $s$ is a $\psi$-suitable $A$-assignment, $\# \in \{+,-\}$ and $m\in \mathbb{N}$ is a \textbf{clock value} (or \textbf{iteration index}) which will, informally speaking, tell how many times the game play can still jump from a looping atom to a label symbol. The game begins from the initial position $(\varphi,r,+,n)$. The game is then played according to rules that contain---in addition to rules which are analogous to the rules used in the evaluation game for $\mathrm{FO}$---the following rules.
    
    \begin{itemize}
        \item In a position $(C_L,s,\#,n)$ we have two principal cases. If $n > 0$, then the game continues from the position $(\rfor(C_L),s,\#,n-1)$. If $n = 0$, then the play of the game ends and neither player wins the play.
        \item In a position $(L\psi,s,\#,n)$, the game simply moves to the position $(\psi,s,\#,n)$.
    \end{itemize}
We then extend this game and thereby define the \textbf{bounded evaluation game} $\mathcal{G}_\omega(\mathfrak{A},r,\varphi)$ as follows.
        The game starts by Abelard picking a natural number $n' \in \mathbb{N}$. After this, Eloise picks a natural number $n \geq n'$. Then the game $\mathcal{G}_n(\mathfrak{A},r,\varphi)$ is played. Eloise wins $\mathcal{G}_\omega(\mathfrak{A},r,\varphi)$ if she wins the game $\mathcal{G}_n(\mathfrak{A},r,\varphi)$, and similarly, Abelard wins if he wins $\mathcal{G}_n(\mathfrak{A},r,\varphi)$. 
        
        If $r$ is the empty assignment $\varnothing$ (and hence $\varphi$ is a sentence), we may write $\mathcal{G}_n(\mathfrak{A},\varphi)$ and $\mathcal{G}_\omega(\mathfrak{A},\varphi)$ instead of $\mathcal{G}_n(\mathfrak{A},\varnothing,\varphi)$ and $\mathcal{G}_\omega(\mathfrak{A},\varnothing,\varphi)$, respectively.
\end{definition}

We note that, concerning the results below, it 
would make no difference if Eloise 
instead of Abelard picked 
her natural number first, and Abelard would 
then pick some greater number. Or the players
could pick their numbers independently, and 
then the larger one of these would be chosen.
However, we shall formally follow the
convention that Abelard picks first.

Both $\mathcal{G}_n(\mathfrak{A},r,\varphi)$
and $\mathcal{G}_{\omega}(\mathfrak{A},r,\varphi)$ are
reachability games for Eloise, and thus, by Theorem \ref{reachabilitygametheorem}, Eloise has a positional winning strategy if and only if she has a general one. The same claims hold for Abelard as well.
The semantics of $\lfo$ is defined as follows.

\begin{definition}
    We define that $\varphi$ is \textbf{true} (or \textbf{boundedly verifiable}) in $\mathfrak{A}$ under $r$, denoted  $\mathfrak{A},r \models_{\omega} \varphi$, iff Eloise has a winning strategy in the game $\mathcal{G}_\omega(\mathfrak{A},r,\varphi)$.
\end{definition}

\begin{definition}
    The notions of \textbf{weak equivalence}
    and \textbf{strong equivalence} are defined for $\lfo$ precisely as for $\ulfo$ in
    Definition \ref{ulfoequivalencedefinition}, 
    but with respect to the bounded evaluation game this time.  
\end{definition}

Somewhat informally, when it is clear from the context that we are considering $\lfo$, we may use the turnstile $\models$ 
instead of $\models_{\omega}$. Furthermore, when $\varphi$ is a
sentence, we may drop $r$ and simply
write $\mathfrak{A}\models \varphi$.

Let $L\psi$ be a subformula of $\varphi$. Now, consider
renaming this $L$ and the corresponding atoms $C_L$ in $L\psi$ in the
strict scope of this particular occurrence of $L$.
Suppose the occurrence of $L$ and the corresponding atoms in its strict scope end up being renamed as $L'$ and $C_{L'}$. This
renaming is \textbf{safe} if $L\psi$ does not
contain free occurrences of $C_{L'}$ and (2) if $\psi$ already contains symbols $L'$, then none of the occurrences $C_L$ to be renamed are in the scope of such occurrences of $L'$.
A \textbf{regularisation} or a formula $\varphi$ is a strongly equivalent (with respect to both $\ulfo$ and $\lfo$) formula obtained from $\varphi$ by safe renamings. This leaves free occurrences of label
symbols as they are.

We already saw that when considering winning strategies for 
Eloise in semantic games, it does not matter whether we limit
attention to positional or general ones, because having a general winning strategy implies
having a positional one (and of course vice versa).
The following theorem generalizes this observation.

\begin{proposition}\label{positionalityall}
    Let $\mathcal{G}$ denote any of
    the semantic games
    $\mathcal{G}_{\infty}(\mathfrak{A},r,\varphi)$,
    $\mathcal{G}_{\omega}(\mathfrak{A},r,\varphi)$,
    $\mathcal{G}_n(\mathfrak{A},r,\varphi)$. Then the following conditions hold. 
    \begin{enumerate}
        \item
        Eloise (respectively, Abelard) has a positional
        winning strategy in $\mathcal{G}$ iff (s)he
        has a general one.
        \item
        Eloise (respectively, Abelard) has a positional
        non-losing strategy in $\mathcal{G}$ iff (s)he
        has a general one.
    \end{enumerate}
\end{proposition}
\begin{proof}
    $\mathcal{G}$ is clearly a reachability  game 
    for both players, and non-losing in $\mathcal{G}$ is
    equivalent to winning the corresponding safety game.
    Thus the claims of the proposition follow immediately 
    from Theorem \ref{reachabilitygametheorem}.
\end{proof}

Due to this proposition, below we will 
almost exclusively consider positional
strategies only, conceived as functions from game
positions to related choices in the game.

When comparing two logics $L_1$ and $L_2$, we 
write $L_1 \leq L_2$ if for every formula $\varphi_1$ of $L_1$, there exists a formula $\varphi_2$ of $L_2$
such that for each suitable model $\mathfrak{A}$
and assignment $s$, the
formula $\varphi_1$ is true in $\mathfrak{A}$ under $s$ iff $\varphi_2$ is true in $\mathfrak{A}$ under $s$. The strict ordering $<$ is defined from $\leq$ in the natural way.

\textbf{Logical consequence} is 
defined in the usual way for both $\lfo$
and $\ulfo$. That is, for a set of
formulas $\Sigma$, we write $\Sigma\models \varphi$
iff for any suitable $\mathfrak{A}$ (i.e., 
any $\mathfrak{A}$ interpreting
the vocabulary in $\Sigma$ and $\varphi$) and
any suitable $s$ (i.e., an assignment 
interpreting the free variables in $\Sigma$ 
and $\varphi$), it holds
that if $\mathfrak{A},s\models\psi$ for all $\psi\in \Sigma$, 
then $\mathfrak{A},s\models \varphi$. A formula $\varphi$ is
said to be $\textbf{valid}$ if $\varnothing\models \varphi$.
We usually drop $\varnothing$ and write $\models\varphi$.
Concerning deduction systems, we
also define the notation $\Sigma\vdash\psi$ in
the usual way. That is, given a 
deduction system $S$, we write $\Sigma\vdash\varphi$ to
mean that $\varphi$ can be inferred in $S$ from the 
premises taken
from $\Sigma$. We write $\vdash\varphi$ to 
denote that $\varnothing\vdash\varphi$. 
The deduction system $S$ is \textbf{complete for validities} if
the implication 
$$\models\varphi
\Rightarrow\ \vdash\varphi$$
holds for
all formulas of the logic investigated. The system $S$ is
\textbf{strongly complete} if
$$\Sigma \models\varphi
\Rightarrow\ \Sigma\vdash\varphi$$
holds for all formulas $\varphi$ and 
any formula set $\Sigma$ of the logic studied.
The system $S$ is \textbf{complete for 
first-order premise sets} if the above 
implication holds in restriction to those 
cases where $\Sigma$ must be a set of
first-order formulas. Note that $\varphi$
does not need to be first-order.
The system $S$ is \textbf{sound} if
$$\Sigma\vdash \varphi\ \Rightarrow\ \Sigma\models\varphi$$
holds for all formulas $\varphi$ and formula sets $\Sigma$ of
the logic investigated.

\section{Properties of $\lfo$}

In this section we prove some basic properties related to $\lfo$.
Now, the \textbf{diameter} of a directed
graph $G = (V,E)$ is 
\[diam(G) = \sup \{d(v,u) \mid v, u \in V\},\]
if the supremum exists, and otherwise it is $\infty$. Here $d(u,v)$ denotes the \textbf{directed distance} between $u$ and $v$ which is formally defined as follows. We let $d(w,w) = 0$, and
for all $w'\in V\setminus \{w\}$, we define $d(w,w')$ to be the smallest number $k+1\in \mathbb{N}_+$
such that we have $d(w,w'') = k$ for some $w''\in V$ such that $Ew''w'$. If no such number $k+1$ exists, then $d(w,w') = \infty$. In this paper, a \textbf{cycle} in a directed graph is a 
finite sequence $(u_1,\dots ,u_k)$ of nodes
such that the following conditions hold.
\begin{enumerate}
\item
 $k>1$
 \item
 $Eu_iu_{i+1}$ for all $i<k$ 
    \item 
    $u_1 = u_k$
\item
$u_i \not= u_j$ for
all distinct indices $i,j\in \{1,\dots , k-1\}$
\end{enumerate}

The following proposition lists examples of properties definable in $\lfo$ which are not
expressible in $\mathrm{FO}$. Showing that these properties are not definable in $\mathrm{FO}$ is an
easy exercise in using Ehrenreucht-Fra\"{i}ss\'{e} games.

\begin{proposition}\label{prop:lfo-expressive-power-examples}
    The following classes are definable in $\lfo$.
    \begin{enumerate}
        \item The class of directed graphs that contain a cycle.
        \item The class of directed graphs with a finite diameter.
    \end{enumerate}
\end{proposition}
\begin{proof}
    To prove the first claim, consider the sentence 
    $$\varphi_{cycle} := \exists x \exists y\bigl(x = y 
    \wedge L(Eyx\vee \exists z(Eyz \wedge \exists y(y = z\wedge C_L))\bigr)$$
    %
    %
    %
    %
    %
    %
    where the purpose of the innermost quantifier $\exists y$ is to redefine $y$ to point at the same element as $z$. Intuitively, this enables us to ``move $y$'' along the nodes of the candidate directed cycle.
    Now, it is easy to check that Eloise has a winning strategy in $\mathcal{G}_\omega(G,\varphi_{cycle})$ iff $G$ contains a cycle.

    To prove the second claim,
    consider the following sentence.
    \[\varphi_{diam} := \forall x \forall y (x=y\ \lor\ L\bigl(Exy \lor \exists z (Exz \land \exists x (x = z \land C_L))\bigr)).\]
    At first sight it might seem that $\varphi_{diam}$
    expresses that there is a directed path between any two vertices in the underlying graph and possibly no finite limit on the lengths of such paths, but since the players are required to declare initial clock values in the beginning of the semantic game for $\lfo$, the formula is actually saying that there is some $n\in \mathbb{N}$ so that between any two vertices, there exists a path of length at most $n$. In other words, it expresses the fact that the diameter of the underlying graph is finite.
\end{proof}

    We remark that since $\lfo$ can expresses the fact that the diameter of the underlying graph is finite, it follows easily that $\lfo$ does not have compactness theorem; otherwise one could use a standard compactness argument to construct a directed graph that is not connected (and hence does not have a finite diameter) but is nevertheless a model of $\varphi_{diam}$.

    The following lemmas are straightforward. 

\begin{lemma}\label{monotonicity}
    Let $\varphi$ be a formula of $\lfo$, 
    and let $n_1,n_2 \in \mathbb{N}$ such that $n_1 < n_2$.
    If Eloise has a winning strategy in  $\mathcal{G}_{n_1}(\mathfrak{A},s,\varphi)$, then she has one also in $\mathcal{G}_{n_2}(\mathfrak{A},s,\varphi)$.
\end{lemma}

\begin{proof}
    It is easy to see that the strategy for $\mathcal{G}_{n_1}$ can be simulated in $\mathcal{G}_{n_2}$. 
    %
    %
    %
\end{proof}

\begin{lemma}\label{existentialgames}
    Let $\varphi$ be a sentence of $\lfo$.
    Now Eloise has a winning strategy in the bounded evaluation game $\mathcal{G}_\omega(\mathfrak{A},s,\varphi)$ iff there exists $n\in \mathbb{N}$ such that Eloise has a winning strategy in the $n$-bounded evaluation game $\mathcal{G}_n(\mathfrak{A},s,\varphi)$.
\end{lemma}
\begin{proof}
Immediate by the definitions and Lemma \ref{monotonicity}.
\end{proof}

By the above Lemma, we have $\mathfrak{A},s\models_{\omega}\varphi$ iff Eloise has a winning strategy in $\mathcal{G}_n(\mathfrak{A},s,\varphi)$
for some $n$. Therefore, if we are only interested in whether or not a formula of $\lfo$ holds in a model $\mathfrak{A}$ under some assignment $s$, we can consider evaluation games where only Eloise declares the initial clock value $n$ and then the game $\mathcal{G}_n(\mathfrak{A},s,\varphi)$ is played. We can ignore the rule that also Abelard
declares a clock value. 
It is then easy to see that any two formulas are weakly equivalent with respect to the standard semantics of $\lfo$ iff they are weakly equivalent with respect to the new semantics. However, it is straightforward to show that the same does not hold for strong equivalence. The following gives a concrete example of how the alternative semantics affects the semantic games.

\begin{example}\label{example:small_clock}
    Consider the sentence
    \[\varphi := \neg \exists x (Px \land L \exists y (Rxy \land (Qy \lor \exists y (y = x \land C_L))),\]
    and let $\mathfrak{A}$ be any suitable model with the property that there exists an $R$-path $\pi$ between the element belonging to $P^\mathfrak{A}$ and an element belonging to $Q^\mathfrak{A}$, and the length (number of edges) of the shortest path from an element in $P^\mathfrak{A}$ to an element in $Q^\mathfrak{A}$ is greater than one. Now, in the game $\mathcal{G}_\omega(\mathfrak{A},\varphi)$, Abelard has the following winning strategy: choose $n'$ so that it is at least, say, the length of $\pi$, and then play along $\pi$ the evaluation game $\mathcal{G}_n(\mathfrak{A},\varphi)$ where $n \geq n'$ is the number selected by Eloise. However, if Eloise is the only player who has influence on the initial clock value, then she has a trivial strategy which guarantees that she does
    not lose the game: set the clock value to zero.
    This shows that while Abelard has a winning strategy in $\mathcal{G}_\omega(\mathfrak{A},\varphi)$, Eloise can easily prevent
    Abelard from winning in the alternative game. Nevetheless, even in the
    alternative game, Eloise surely has no winning strategy, as having one would
    imply she has one also in the standard game for $\lfo$.
\end{example}


Finally, note that our standard semantics for $\lfo$ has the 
property that $\mathfrak{A},s\models
\neg\varphi$ iff the (opponent 
player) Abelard has a winning strategy in $\mathcal{G}_{\omega}(\mathfrak{A},s,\varphi)$. This is a 
desirable property that gives an interpretation to negation,
and this property does not
hold as such in the alternative
semantics where Abelard does not
declare clock values.

\subsection{Approximants of formulas}\label{approximantssection}

A standard technique in the study of logics with fixed points (such as the modal $\mu$-calculus) is the use of approximants, which evaluate the fixed point only up to some fixed bound. Even though $\lfo$ is not really based on fixed points in any direct way, a natural notion of an
approximant can be defined also for $\lfo$. However, this notion differs from the corresponding notion for fixed point logics in a number of crucial ways. For example, we will have to adjust the approximants to take care of the number of negations
encountered in related plays of the semantic game.
These issues will become transparent
below when we give the related formal definitions.
Intuitively, the $n$th approximant for a formula of $\lfo$ will describe an evaluation game where the initial clock value is set to $n$. Before giving the definition of approximants, we provide the following auxiliary definition.

\begin{definition}\label{unfoldingdefinition}
Let $\varphi \in \lfo$. If $\varphi$ is
not a regular formula, we let $\varphi'$ denote a 
regularisation of it. To keep the current definition
deterministic, we suppose $\varphi'$ is obtained
from $\varphi$ in some systematic way. 
If $\varphi$ is already regular then $\varphi' = \varphi$. 
Now, the \textbf{$n$th
unfolding} (or \textbf{$n$-unfolding}) of $\varphi$, denoted by $\Psi_{\varphi}^n$, is defined inductively as follows. 
\begin{enumerate}
    \item 
The zeroeth
unfolding $\Psi_{\varphi}^{0}$ of $\varphi$ is
defined to be the formula $\varphi'$.
\item
The $(k+1)$st unfolding $\Psi_{\varphi}^{k+1}$ is the formula obtained from the $k$th unfolding $\Psi_{\varphi}^k$
by replacing every looping atom $C_{L}$ in $\Psi_{\varphi}^k$ by the corresponding 
reference formula $\mathrm{Rf}(C_{L})$ in $\Psi_{\varphi}^k$ (if 
the reference formula exists, i.e., if $C_L$ is
not free).
\end{enumerate}
%
%
%
\end{definition}

We now define the notion of an approximant.

\begin{definition}
%
%
%
Let $\varphi$ be a formula of $\lfo$. We define the \textbf{$n$th approximant} (or \textbf{$n$-approximant}) $\Phi_{\varphi}^n$ of $\varphi$ to be the $\mathrm{FO}$-formula
obtained from the $n$th unfolding $\Psi_{\varphi}^n$ by removing all the label symbols and replacing each occurrence of each looping atom by
\begin{enumerate}
    \item 
    $\bot$ if the occurrence of the atom is positive in $\Psi_{\varphi}^n$,
    \item
    $\top$ if the occurrence is negative in $\Psi_{\varphi}^n$.
\end{enumerate}
%
%
%
%
%
%
\end{definition}

\begin{example}
    Consider the sentence $LC_L \in \lfo$. Every approximant of this sentence is just $\bot$.
\end{example}

\begin{example}
    Consider the sentence 
    $$\varphi_{cycle} := \exists x \exists y\bigl(x = y 
    \wedge L(Eyx\vee \exists z(Eyz
    \wedge \exists y(y = z\wedge C_L))\bigr)$$
    of $\lfo$ from Proposition \ref{prop:lfo-expressive-power-examples} that defines the class of directed graphs that contain a cycle.
    Its first two approximants are
    $$\Phi^0\ =\ \exists x \exists y\bigl(x = y 
    \wedge (Eyx\vee \exists z(Eyz \wedge \exists y(y = z\wedge \bot))\bigr)$$
    and
    \begin{align*} 
    \Phi^1\ & =\ \exists x \exists y\bigl(x = y 
    \wedge (Eyx\vee \exists z(Eyz \wedge \exists y(y = z
    \wedge \chi\, ))\bigr)\text{ where}\\
    \chi &=  Eyx\vee \exists z(Eyz \wedge \exists y(y = z
    \wedge \bot ).
    \end{align*}
    The zeroeth approximant expresses that there exists a reflexive loop (i.e., a cycle of length one), while the first approximant asserts that there exists a cycle of length at most two. In general, the $n$th approximant expreses that there exists a cycle of length at most $n+1$.
\end{example}

The following lemma gives a natural 
characterization of approximants.

\begin{lemma}\label{approximationswork}
    Let $\varphi \in \lfo$ be a formula.
    Now Eloise has a winning strategy in the $n$-bounded game $\mathcal{G}_n(\mathfrak{A},s,\varphi)$ iff she has a winning strategy in $\mathcal{G}(\mathfrak{A},s,\Phi_{\varphi}^n)$.
\end{lemma}
\begin{proof}
First notice that by the construction of $\Phi_{\varphi}^n$,
the game trees of the games $\mathcal{G}_n(\mathfrak{A},s,\varphi)$
and $\mathcal{G}(\mathfrak{A},s,\Phi_{\varphi}^n)$
are essentially identical all the way from the root to the level with
positions of type $(C_L,r,+,0)$ and $(C_L,r,-,0)$ in $\mathcal{G}_n(\mathfrak{A},s,\varphi)$.
The 
position in $\mathcal{G}(\mathfrak{A},s,\Phi_{\varphi}^n)$
that corresponds to the position $(C_L,r,+,0)$ of 
$\mathcal{G}_n(\mathfrak{A},s,\varphi)$ is $(\bot,r,+)$, and similarly, the position
corresponding to $(C_L,r,-,0)$ is $(\top,r,-)$. Eloise would not win any play entering any of such positions. 
Therefore no winning strategy of Eloise will lead to a play that enters such positions.

Now, if Eloise has a winning strategy in
one of the games $\mathcal{G}_n(\mathfrak{A},s,\varphi)$
and $\mathcal{G}(\mathfrak{A},s,\Phi_{\varphi}^n)$, she
can simulate that strategy in the other game. Each play where Eloise follows a winning strategy will end with her winning in some position
involving some $\mathrm{FO}$-atom $\delta$. The simulated play in the other game will then, likewise, end with her win in a position with $\delta$. Notice that when simulating a winning strategy for $\mathcal{G}(\mathfrak{A},s,\Phi_{\varphi}^n)$ in order to 
win $\mathcal{G}_n(\mathfrak{A},s,\varphi)$, every non-winning position $(C_L,r,+,0)$ (respectively, $(C_L,r,-,0)$) will indeed be avoided because Eloise will not enter the corresponding position $(\bot,r,+)$ (resp., $(\top,r,-)$) in $\mathcal{G}(\mathfrak{A},s,\Phi_{\varphi}^n)$
because such a position would be losing for her.
A corresponding principle holds in the direction where Eloise simulates a strategy from $\mathcal{G}_n(\mathfrak{A},s,\varphi)$ to win a game on
the approximant. 
\end{proof}

The following lemma is immediate.

\begin{lemma}\label{indexmonotonicity}
   If $\mathfrak{A}\models \Phi_{\varphi}^n$, 
   then $\mathfrak{A}\models \Phi_{\varphi}^{n'}$
   for all $n' \geq n$.
\end{lemma}

\begin{proof}
    Combine Lemma \ref{approximationswork} with Lemma \ref{monotonicity}.
\end{proof}

\subsection{Applications of approximants}

We will now show that $\lfo$ translates into $\mathcal{L}_{\omega_1\omega}$, i.e., the
extension of $\mathrm{FO}$ that allows for countably infinite conjunctions and disjunctions.

\begin{theorem}\label{approximationsequivalent}
    Let $\varphi$ be a formula of $\lfo$. 
    Then we have 
    $$\mathfrak{A},s \models_{\omega} \varphi\ \Leftrightarrow\  \mathfrak{A},s \models \bigvee_{n\in \mathbb{N}} \Phi_{\varphi}^n.$$
\end{theorem}
\begin{proof}
    By Lemma \ref{existentialgames}, we have $\mathfrak{A},s \models_{\omega} \varphi$ iff there exists some $n\in \mathbb{N}$ such that Eloise has a winning strategy in  $\mathcal{G}_n(\mathfrak{A},s,\varphi)$.
    By Lemma \ref{approximationswork}, the latter condition is equivalent to there existing some $n$ such that $\mathfrak{A} \models \Phi_{\varphi}^n$.
\end{proof}

We can use theorem \ref{approximationsequivalent} to prove undefinability results for $\lfo$, and as an example of this, we next prove that graph connectivity is not definable in $\lfo$.

\begin{proposition}
    The class of (finite and infinitary) connected graphs is not definable in $\lfo$.
\end{proposition}
\begin{proof}
    Let 
    \[\mathfrak{G}_1 = (\mathbb{Z},\{(n,n+1) \mid n \in \mathbb{Z}\} \cup \{(n+1,n) \mid n\in \mathbb{Z}\}),\]
    and let $\mathfrak{G}_2$ be the disjoint union of two copies of $\mathfrak{G}_1$. Suppose that there exists a sentence $\varphi \in \lfo$ which is true in a graph iff the graph is connected. Then $\mathfrak{G}_1 \models_{\omega} \varphi$ and hence $\mathfrak{G}_1 \models \Phi_{\varphi}^n$ for some $n\in \mathbb{N}$. Now, it is straighforward to show that $\mathfrak{G}_1$ and $\mathfrak{G}_2$ are elementarily equivalent (for example using Ehrenfeucht-Fra\"{i}ss\'{e} games). Thus we have $\mathfrak{G}_2 \models \Phi_\varphi^n$, which implies that $\mathfrak{G}_2 \models_{\omega} \varphi$ by Theorem \ref{approximationsequivalent}. This is a contradiction.
\end{proof}

We then relate validities of $\lfo$ to
approximants.

\begin{theorem}\label{LFOVALIDITY}
    Let $\varphi \in \lfo$ be a formula. Now $\varphi$ is valid iff $\Phi_{\varphi}^n$ is valid for some $n\in \mathbb{N}$.
\end{theorem}
\begin{proof}
    The direction from right to left follows directly from Theorem \ref{approximationsequivalent}.
    Suppose then that none of the sentences $\Phi_{\varphi}^n$ is valid. We claim that then
    \[\Sigma := \{\neg \Phi_{\varphi}^n \mid n\in \mathbb{N}\}\]
    is satisfiable. By compactness, it suffices to show that all finite subsets of $\Sigma$ are satisfiable. But this is clear, since each sentence $\neg \Phi_\varphi^n$ is now  satisfiable (due to no $\Phi_\varphi^n$ being valid) and we have $\neg \Phi_\varphi^n \models \neg \Phi^{n'}$ for all $n' < n$ by Lemma \ref{indexmonotonicity}. Since $\Sigma$ is satisfiable, we have $\mathfrak{A},s \models \Sigma$ for some $\mathfrak{A}$ and $s$. Thus Eloise cannot have a winning strategy in  $\mathcal{G}_n(\mathfrak{A},s,\varphi)$ for any $n\in \mathbb{N}$ by Lemma \ref{approximationswork}. Therefore $\varphi$ is not valid.
\end{proof}

Since the set of valid $\mathrm{FO}$-sentences is recursively enumerable, the following corollary is immediate.

\begin{corollary}
    The set of valid formulas of $\lfo$ is recursively enumerable.
\end{corollary}

Since $\lfo$ is not closed under contradictory negation, the above corollary leaves open the complexity of its satisfiability problem. Using approximants, we can determine the exact complexity of the satisfiability problem of $\lfo$.
Before that, we give some auxiliary definitions.

The \textbf{grid} is the structure $\mathfrak{G}\, :=\,  (\mathbb{N}\times\mathbb{N}, H , V)$ where the binary relations 
    $$H := \{\, ((i,j),(i+1,j))\ |\ i,j\in \mathbb{N}\, \}$$
and $$V := \{\, ((i,j),(i,j+1))\ |\ i,j\in \mathbb{N}\, \}$$
are suggestively called the \textbf{horizontal} and 
\textbf{vertical successor relation}. By $\mathfrak{G}^*$ we denote the expansion $(\mathbb{N}\times\mathbb{N}, H , V,H^*,V^*)$ of $\mathfrak{G}^*$ where $H^*$ and $V^*$ are the transitive closures of $H$ and $V$.

Now consider a deterministic Turing machine $M$ with one-way infinite tape and one read-write head. Denote the set of states of $M$ by $Q$.
Let $S_t$ and $S_i$ be, respectively, the sets of tape and input symbols of $M$, and denote $S := S_t\cup S_i$. Let $s\in S_i^*$ be an input to $M$. Let $\tau$ denote the vocabulary that contains the relation symbols of $\mathfrak{G}^*$ and additionally a unary relation symbol $P_u$ 
for each $u\in S\cup Q$.
The \textbf{computation table} $\mathfrak{T}_{M,s}$ of $M$ with the input $s$ is
the expansion of $\mathfrak{G}^*$ defined as follows.
\begin{enumerate}
    \item 
    Consider the point $(i,j)$ of $\mathfrak{T}_{M,s}$.
    For each predicate $P_u\in S$, we define $P_u$ to be
    true at the point $(i,j)$ of $\mathfrak{T}_{M,s}$ iff at time $j$, the tape cell $i$ contains the symbol $u$.
    Thus, intutively, the tape of $M$ at time $j$ is encoded by the $j$th row of the structure $\mathfrak{T}_{M,s}$. Note that the input $s$ is of course encoded to the beginning cells of row $0$.
    \item
    We define $P_u\in Q$ to be true at the point $(i,j)$ of $\mathfrak{T}_{M,s}$ iff at time $j$, the read-write head of $M$ is at the cell $i$ and the
    current state of $M$ is $u$.
\end{enumerate}

A \textbf{generalized ordered grid} is a structure 
$\mathfrak{T}\, :=\,  (A\times B, H , V, H^*, V^*)$ 
such that the following conditions hold.
\begin{enumerate}
    \item 
    The sets $A$ and $B$ are domains of two discrete order structures $(A,<^A)$ and $(B,<^B)$. Both $(A,<^A)$ and $(B,<^B)$ have a mininum element. The relation $H \subseteq (A\times B)\times (A\times B)$ is the \textbf{horizontal successor relation}
    $$H\, :=\, \{ \bigl((a,b),(a',b)\bigr)\ |\   
    b\in B\text{ and }a'\text{ is the }<^A\text{-successor of }a\, \}. $$
    Similarly, $V$ is the \textbf{vertical successor relation}
    $$V\, :=\, \{ \bigl((a,b),(a,b')\bigr)\ |\   
    a\in A\text{ and }b'\text{ is the }<^B\text{-successor of }b\, \}. $$
    \item
    $H^*$ is the relation $$\{ \bigl((a,b),(a',b)\bigr)\ |\   
    b\in B\text{ and }a <^A a'\}$$
    and $V^*$ the relation $$\{ \bigl((a,b),(a,b')\bigr)\ |\   
    a\in A\text{ and }b <^B b'\}.$$
\end{enumerate}

Our next aim is to define the notion of a \textbf{generalized computation table}. These are
structures that resemble computation tables but are built on generalised ordered grids.
To define the notion, we let $M$ denote, as above, a deterministic Turing machine with a one-way infinite tape and one read-write head. The set of
states of $M$ is denoted by $Q$. We let $S_t$ and $S_i$ be the sets of tape and input symbols of $M$, and we define $S := S_t\cup S_i$. We let $\tau$ be the vocabulary containing the relation symbols of generalized ordered grids and additionally a unary relation symbol $P_u$ for each $u\in S\cup Q$.

Now, consider an infinite discrete order
structure $(A,<^A)$ with a minimum element $0^A$. Suppose $A'\subseteq A$ is a prefix set for $(A,<^A)$. Then a function $s : A'\rightarrow S_i$ is a \textbf{generalized input} for $M$ and $(A,<^A)$, the intuition being that $s$ labels some prefix of $(A,<^A)$ with the input symbols of $M$. Let $(B,<^B)$ be an infinite discrete order structure with a minimum element $0^B$.
Now consider the 
structure $\mathfrak{T}_{M,s}$ defined as follows.
\begin{enumerate}
    \item 
    $\mathfrak{T}_{M,s}$ expands 
    the generalized grid $\mathfrak{T} = (A\times B,H,V,H^*,V^*)$ to the vocabulary $\tau$.
    \item
    For each $b\in B$, there exists precisely one cell $(a,b)\in A \times B$ such that for some $u\in Q$, the predicate $P_u$ is satisfied at $(a,b)$. Furthermore, if some $P_u$ is satisfied at the cell $(a,b)$, then no other $P_{u'}$ with $u\in Q$ is satisfied at that cell. Intuitively, all this simply means that
    the following two conditions hold.
    \begin{itemize}
        \item At the computation stage indexed by (the row) $b\in B$, the read-write head is located in the cell $a\in A$.
        \item
        At that computation stage, the machine $M$ is in state $u\in Q$.
    \end{itemize} 
    \item
    The first cells of the row $0^B$ are indexed according to the generalized input $s$ in the natural way. In other words, the (possibly infinitary) input to $M$ is as given by the generalized input $s$.
    \item
    At each point $(a,b)$ of the structure excluding row zero, the truth of the predicates $P_u$ with $u\in S$ are determined by neighbouring cells in the previous row, i.e., the cells $(a',b')$, $(a,b')$, $(a'',b')$ such that 
    $(a,b')V(a,b)$ and $(a',b')H(a,b')$
    and $(a,b')H(a'',b')$. The truth of the predicates is determined in the natural way according to the computation of $M$. Note that these predicates can change their truth value only if the predicates indicating the position of the read-write head is in the vicinity. The predicates corresponding to the read-write head are similarly locally related to the computation of the machine and relate to the other symbols in the correct way as allowed by $M$. The read-write head begins from the first cell in the beginning of computation. 
\end{enumerate}
Such a structure is a \textbf{generalized computation table} for $M$ and $s$. A generalized computation table for $M$ is a structure that is a generalized computation table for $M$ and some generalized input $s$.

\begin{lemma}\label{tablelemma}
    Let $M$ be a deterministic one-way Turing machine with one read-write head.
    The class of generalized computation tables for $M$ is definable by an $\mathrm{FO}$-sentence.
\end{lemma}

We are now ready to prove the following.

\begin{theorem}\label{prop:lfo-sat}
    The satisfiability problem for $\lfo$ is $\Sigma_2^0$-complete.
\end{theorem}
\begin{proof}
    We prove the upper bound first. By Lemma \ref{approximationsequivalent}, a given sentence $\varphi$ of $\lfo$ is satisfiable iff there exists some $n$ such that the approximant $\Phi_\varphi^n$ is satisfiable. Therefore, and since the satisfiability problem for $\fo$ is in $\Pi_1^0$, satisfiability for $\lfo$ can be
    defined in $\Sigma_2^0$.

Next we shift our attention to the lower bound.
Given a Turing machine $M$, we use $L(M)$ to denote the set of strings accepted by $M$. The following problem is known to be $\Sigma_2^0$-hard \cite{KOZENBOOK}: \emph{given a Turing machine $M$, determine whether $L(M)$ is finite}. To prove the $\Sigma_2^0$-hardness of the satisfiability problem of $\lfo$, we give a recursive mapping $M \mapsto \varphi_M$ such that the $\lfo$-sentence $\varphi_M$ is satisfiable iff $L(M)$ is finite.

    Fix a Turing machine $M$. Without loss of generality, we may assume that the vocabulary of $M$ is $\{0,1\}$. To simplify the construction of $\varphi_M$, instead of $M$ we consider a Turing machine $M^*$ which, when given
    \[1^n = \underbrace{1 \dots 1}_{\text{$n$-times}},\]
    as an input a string,  
    does the following.
    \begin{itemize}
    \item
    $M^*$ begins the process of enumerating the set of binary strings of length at least $n$.
    \item
    During the enumeration process, $M^*$ halts if it encounters a string accepted by $M$.
    \end{itemize}
    Clearly
    \[L(M) \cap \{0,1\}^{\geq n} = \varnothing\ \Leftrightarrow\ M^* \text{ does not halt on the input $1^n$}.\]
    Thus
    \[L(M) \text{ is finite}\ \Leftrightarrow\ \text{for some $n\geq 0$, $M^*$ does not halt on the input $1^n$.}\]
    The sentence $\varphi_M$ will express the right
    hand side of this equivalence.

The sentence $\varphi_{M}$ will be true in precisely 
those models that are isomorphic to 
some generalized
computation table $\mathfrak{T}_{M^*,1^n}$ that
encodes a non-halting computation with the input being a finite
string on bits $1$. 
Now, by Lemma \ref{tablelemma}, there
exists an $\mathrm{FO}$-formula $\psi_{M^*}$
that defines the class of generalized computation tables
for $M^*$. The formula $\varphi_M$ will be a conjunction
$$\psi_{M^*} \wedge \psi' \wedge \chi$$
such that the following conditions hold.
\begin{enumerate}
    \item 
    $\psi'$ is an $\mathrm{FO}$-sentence making
    sure that the computation of $M^*$ does not halt.
    %
    %
    %
    \item
    $\chi$ is a sentence of $\lfo$ asserting
    that there exists some $n\in \mathbb{N}$ such that the input to the computation is $1^n$.
\end{enumerate}
While writing $\psi'$ in $\mathrm{FO}$ is straightforward, 
the attempt to define $\chi$ in $\mathrm{FO}$ (rather than $\lfo$)
runs into the challenge of specifying that only
finite input strings $1^n$ (of
arbitrary lengths) are allowed. Thereby we resort to the expressive capacities of $\lfo$ for defining $\chi$. 
Now, to define $\chi$, Recall that $P_0$ 
and $P_1$ are unary predicates that
encode the bits $0$ and $1$,
respectively. We first give the following auxiliary formulas. 
    \begin{itemize}
        \item
        $\mathit{column\_zero}(x)\ :=\ \neg\exists y H^*yx$
        \item
        $\mathit{row\_zero}(x)\ :=\ \neg\exists y V^*yx$
        \item
        $\chi_0\ :=\ \neg\exists x(\mathit{row\_zero}(x) \wedge P_0(x))$
    \end{itemize}
    We then define the formula $\chi$ as follows.
    \begin{align*}
        \chi\ 
    :=\ \chi_0\ \wedge\ \forall x\bigl(\ & \neg \mathit{row\_zero}(x)
    \vee\ \neg P_1(x)\\
    \vee\ &L\bigl(\mathit{column\_zero}(x)\ \vee\ 
    \exists y(Hyx \wedge \exists x(x = y \wedge C_L))\bigr).
    \end{align*}
    It is easy to show that $\varphi_M$ is as required.
\end{proof}

\section{Results on $\ulfo$}

In this section we investigate $\ulfo$.
We first list some interesting properties that are definable in $\ulfo$.

\begin{proposition}\label{ulfoexample}
    The following classes of models are definable in $\ulfo$.
    \begin{enumerate}
        \item The class of connected graphs (i.e., graphs such that for all vertices $u$ and $v\not= u$, there is a directed path from $u$ to $v$). We do not limit attention to finite graphs here. 
        \item The class of well-founded linear order
        structures.
    \end{enumerate}
\end{proposition}
\begin{proof}
    To establish the first claim, we notice that clearly the following sentence of $\ulfo$ is true in a graph $(V,E)$ iff the graph is connected:
    $$\forall x \forall y \bigl(x=y\ \vee\ L(Exy \lor \exists z (Exz \land \exists x (x = z \land C_L)))\bigr).$$
    %
    %
    %

    To establish the second claim of the theorem, we claim that the sentence
    $$\forall x L \forall y \bigl(\neg y < x\, \vee\, \forall x(\neg x = y \lor C_L)\bigr)$$
    of $\ulfo$ is true in a linear order structure $(A,<)$ iff the linear order is well-founded.
    Indeed, it is easy to see that if $(A,<)$ is well-founded, then the game can always be won by Eloise. On the other hand, if $(A,<)$ is not well-founded, then Abelard can force the play of the game to last for infinitely many rounds.
\end{proof}

\subsection{The validity problem of $\ulfo$}

Our next aim is to prove that the set of validities of $\ulfo$ is recursively enumerable. As in the case of $\lfo$, our arguments make extensive use of approximants introduced for $\lfo$ (as opposed to $\ulfo$). As the following example demonstrates, we will have to use these approximants in a slightly more sophisticated way.

\begin{example}
Consider the sentence $$\varphi = \forall x L \forall y \bigl(\neg y < x\, \vee\, \forall x(\neg x = y \lor C_L)\bigr)$$ from Proposition \ref{ulfoexample}. The proposition
shows that this sentence is true in a
linear order $(A,<)$
according to $\ulfo$ iff the linear
order is well-founded. Now, the sentence $\neg \Phi_{\varphi}^n$, where $\Phi_{\varphi}^n$ is the $n$-approximant of $\varphi$, is true on a linear order structure $(A,<)$ iff $<$ has a decreasing sequence of at least $n+1$ nodes. Hence $(\mathbb{N},<)$ is a model for the set $\{\varphi\} \cup \{\neg \Phi_{\varphi}^n\ |\ n \in \mathbb{N}\}$
according to $\ulfo$. 
\end{example}

The above example demonstrates that there are
sentences $\varphi$ of the logic $\ulfo$ and models $\mathfrak{A}$ for which Eloise has a winning strategy in the game $\mathcal{G}_\infty(\mathfrak{A},\varphi)$, but not in any of the games $\mathcal{G}_n(\mathfrak{A},\varphi)$. Nevertheless, we can still establish the following compactness-like property for $\ulfo$.

\begin{lemma}\label{compactness}
    Let $\varphi$ be a formula of $\ulfo$. Suppose that for every $n\in \mathbb{N}$, there exists a model $\mathfrak{A}_n$ and an assignment $s_n$ so that Eloise does not have a winning strategy in the game $\mathcal{G}_n(\mathfrak{A}_n,s_n,\varphi)$. Then there exists a model $\mathfrak{A}$ and an assignment $s$ so that Eloise does not have a winning strategy in the game $\mathcal{G}_\infty(\mathfrak{A},s,\varphi)$.
\end{lemma}
\begin{proof}
We may assume that the domains of the 
assignments contain precisely the set of 
variables occurring in $\varphi$, free
and bound. Thus, we can clearly even assume
that the assignments have the
domain $\varnothing$ by extending
the underlying vocabulary by 
finitely many constant symbols. We note that we
strictly speaking consider purely relational 
vocabularies, so these constants ultimately need to be
encoded by relation symbols. We shall below further
extend our vocabulary also by function symbols, and these also need to be encoded by relation symbols. All these 
encodings, however, are
straightforward, so we can make the related assumptions
without further discussion.
Now, as $s=\varnothing$, the
formula $\varphi$ is a sentence.
We may also assume that $\varphi$
makes no use of $\vee$ or $\forall$, as
these can be defined via using 
the operators $\neg$, $\wedge$, $\exists$.
Finally, we may assume, without loss of
generality, that $\varphi$ does not have 
free looping atoms, as
such an atom can be 
replaced by $LC_L$.

Now, suppose indeed that for every $n\in \mathbb{N}$, there exists a model $\mathfrak{A}_n$ so that Eloise does not have a winning strategy in the game $\mathcal{G}_n(\mathfrak{A}_n,\varphi)$.
    Based on this, we will construct a countable first-order theory $\Sigma$ such that from any model of $\Sigma$, we can read off a model $\mathfrak{A}$ and a non-losing  strategy $\sigma$ for Abelard in $\mathcal{G}_\infty(\mathfrak{A},\varphi)$.
    This implies, in particular, that Eloise does not have a winning strategy in the game $\mathcal{G}_\infty(\mathfrak{A},\varphi)$. To show that the theory $\Sigma$ is satisfiable, we will use compactness of $\mathrm{FO}$.

    We will start defining $\Sigma$ by first specifying the underlying vocabulary $\tau$ which will be finite but extend the vocabulary of $\varphi$. So, we first put into $\tau$ every relation symbol occurring in $\varphi$. We assume that $\{x_1,\dots ,x_k\}$ is the set of all
    variables that occur in the sentence $\varphi$. 
    We then define the set
    \[S := \{N_{(\psi, \#)} \mid \psi \in \subf(\varphi), \# \in \{+,-\}\} \cup \{D,W, <\},\]
    of relation symbols,
    where each of the symbols $N_{(\psi,\#)}$ is a $(k+1)$-ary relation symbol, $<$ is a binary relation symbol and
    $W$ and $D$ are unary relation symbols.
    We add the symbols in $S$ into $\tau$. Note that the subscript $(\psi,\#)$ of $N_{(\psi,\#)}$ encodes partial information on positions of evaluation games for $\varphi$, namely, $(\psi,\#)$ lists the current subformula occurrence $\psi$ being played and $\#\in\{+,-\}$ encodes whether Eloise or Abelard is
    currently the verifier in the game.\footnote{Syntactically identical subformulas that are in different parts of $\varphi$ can be distinguished by some convention whose details are not of importance here.} Each predicate $N_{(\psi,\#)}$ has arity $k+1$, and the first $k$ elements of a tuple of $N_{(\psi,\#)}$ encode the current assignment in the evaluation game while the $(k+1)$st element will encode (together with $<$) a measure of how many times positions with claim atoms $C_L$ have been visited. The symbol $D$ will encode the domain of the models $\mathfrak{A}_n$ and $W$ the (disjoint) domain of $<$. The details will be formalized below. To conclude the definition of
    $\tau$, we will add the constant symbol $d$ and
    the unary function symbol $f$ to $\tau$.

    We then proceed with the definition of the theory $\Sigma$. First we need a sentence
    saying that $<$ is a discrete linear order over the set encoded by the unary predicte $W$, and that $d$ is the minimum element of this ordering. We will also need the sentence
    \begin{align*}
        \forall x \bigl(\exists y (x < y)\ 
        \to\ \bigl(x < f(x) 
        \land \forall y (x < y \to (f(x) < y \lor y = f(x)))\bigr)\bigr)
    \end{align*}
    which expresses that $f$ maps each element to its immediate successor with respect to $<$, if such an element exists. 
    Furthermore, we need a sentence ensuring that the intersection of $D$ and $W$ is empty.

    We then write axioms for the relations $N_{(\psi,\#)}$. The intuition is that tuples in the interpretation of $N_{(\psi,\#)}$ encode non-losing positions for Abelard in the evaluation game, i.e., positions that are not winning for Eloise. Thus the axioms of $\Sigma$ will encode natural safety restrictions on the evaluation game.

    First, for every $\mathrm{FO}$-atom $\alpha := \alpha(x_{i_1},\dots ,x_{i_j})$
    with variables in $\{x_1,\dots ,x_k\}$, where $(x_{i_1},\dots ,x_{i_j})$ lists the variables of $\alpha$, we add the sentences
    \[\forall x_1 \dots x_k \forall c (N_{(\alpha,+)}(x_1, \dots ,x_k,c) \to \neg \alpha(x_{i_1}, \dots ,x_{i_j}))\]
    and
    \[\forall x_1\dots x_k \forall c (N_{(\alpha, -)}(x_1, \dots ,x_k,c) \to \alpha(x_{i_1}, \dots ,x_{i_j}))\]
    to the theory $\Sigma$. Such sentences guarantee that $N_{(\alpha,+)}$ and $N_{(\alpha,-)}$ will not contain positions involving atoms such that Abelard directly loses the
    evaluation game in those positions. For every subformula $\psi \land \chi$, we add to $\Sigma$ the formulas
    \begin{align*}
        \forall x_1\dots x_k\forall c & (N_{(\psi \land \chi, +)}(x_1,\dots ,x_k,c)\\
        & \to (N_{(\psi,+)}(x_1,\dots ,x_k,c) \vee N_{(\chi,+)}(x_1,\dots ,x_k,c)))
    \end{align*}
    and
    \begin{align*}
    \forall x_1\dots x_k\forall c & (N_{(\psi \land \chi, -)}(x_1,\dots,x_k,c)\\
    & \to (N_{(\psi,-)}(x_1,\dots,x_k,c) \wedge N_{(\chi,-)}(x_1,\dots,x_k,c))).
    \end{align*}
    For every subformula $\exists x_i \psi$, we add the following formulas
    \begin{align*}\forall x_1\dots x_k \forall c &   (N_{(\exists x_i \psi, +)}(x_1,\dots,x_k,c)\\
    & \to \forall x_i (
    D(x_i) \to 
    N_{(\psi,+)}(x_1,\dots,x_k,c)))
    \end{align*}
    and
    \begin{align*}\forall x_1\dots x_k \forall c & (N_{(\exists x_i \psi, -)}(x_1,\dots,x_k,c)\\
    & \to
    \exists x_i (
    D(x_i) \land
    N_{(\psi,-)}(x_1,\dots , x_k,c)))
    \end{align*}
    to the theory $\Sigma$. To cover negation, we
    add to $\Sigma$ the formulas
        \begin{align*}\forall x_1\dots x_k \forall c   (N_{(\neg \psi, +)}(x_1,\dots,x_k,c)
     \to (
    N_{(\psi,-)}(x_1,\dots,x_k,c)))
    \end{align*}
    and
    \begin{align*}\forall x_1\dots x_k \forall c  (N_{(\neg \psi, -)}(x_1,\dots,x_k,c)
    \to (
    N_{(\psi,+)}(x_1,\dots , x_k,c))).
    \end{align*}

    We will also need two axioms for claim symbols $C_L$. For every subformula $C_L$ with
    the reference formula $L\psi$, we add to $\Sigma$, for every $\# \in \{+,-\}$, the formula
    \begin{align*}
        \forall x_1 \dots x_k \forall c &   \bigl((N_{(C_L,\#)}(x_1,\dots ,x_k,c)
        \wedge \exists y (c < y))\\
        & \rightarrow\ \exists z (W(z) \wedge f(c) = z \wedge N_{(L \psi, \#)}(x_1,\dots,x_k,z))\bigr).
    \end{align*}
    Similarly, for every subformula $L\psi$
    and $\#\in \{+,-\}$, we add to $\Sigma$ the formula
    \[\forall x_1\dots x_k \forall c (N_{(L \psi, \#)}(x_1,\dots,x_k,c) \to N_{(\psi,\#)}(x_1,\dots,x_k,c)).\]

    We will also need the following formula which encodes the initial position of the game.
    \[\exists x_1\dots x_k N_{(\varphi,+)}(x_1,\dots,x_k,d).\]

    Now, if $\Phi$ denotes the (finite) set of axioms that we have listed above, then as our theory $\Sigma$ we choose the following set
    \[\Phi \cup \{\theta_n^d \mid n\in \mathbb{N}\}\]
    where each $\theta_n^d$ expresses
    that $W$ has at least $n$ elements $u$ such that $d  < u$. Therefore, in any model of $\Sigma$, the sequence 
    $$d, f(d), f(f(d)), \dots $$
    is an infinite ascending sequence of
    elements with 
    respect to $<$.

    If $\Sigma$ is satisfiable in a model $\mathfrak{A}^+$, then, in the submodel of $\mathfrak{A}^+$ induced by the set $D^{\mathfrak{A}^+}$, Abelard can survive the semantic game for $\varphi$ without losing. This submodel (or its reduct to the vocabulary of $\varphi$) is then the model $\mathfrak{A}$ required by the current theorem.

    To show that $\Sigma$ is indeed satisfiable, it suffices, by compactness of $\mathrm{FO}$,
    to show that each finite subset of $\Sigma$ is satisfiable. To this end, it clearly suffices to show that for every $n \in \mathbb{N}$, the set
    \[\Sigma_n := \Phi \cup \{\theta_m^d \mid 1\leq m \leq n\}\]
    is satisfiable. Now, we have assumed that there exists a model $\mathfrak{A}_n$ for each $n\in\mathbb{N}$ such that Eloise does not have a winning strategy in the game $\mathcal{G}_n(\mathfrak{A}_n,\varphi)$. 
    Thus, by Theorem \ref{reachabilitygametheorem}, Abelard has a positional non-losing strategy in $\mathcal{G}_n(\mathfrak{A}_n,\varphi)$. We create a model of $\Sigma_n$ as follows.
    \begin{enumerate}
        \item
        We take a copy of $\mathfrak{A}_n$ and interpret $D$ to correspond to the domain of $\mathfrak{A}_n$.
        \item
        We take a disjoint (from $D$) set of size at least $n+1$ and interpret $W$ to correspond to that set. We also interpret the symbols $<$, $d$, $f$ over the set $W$ so that $d$ becomes the minimum element of the order $<$ and $f$ runs step by step from $d$
        upwards along $<$.
        \item
        The predicates $N_{(\psi,\#)}$ are interpreted in the natural way according to the non-losing strategy of Abelard over $\mathfrak{A}_n$. Note that the $(k+1)$st elements of the tuples of $N_{(\psi,\#)}$ run
        along the order $<$.
    \end{enumerate}
Thereby we create a model for $\Sigma_n$ for each $n\in \mathbb{N}$, as required. 
\end{proof}

By the above lemma, if a formula of $\ulfo$ is valid, then already one of the approximants $\Phi_\varphi^n$ is. It is easy to see that also the converse holds.

\begin{lemma}\label{monotonicityinfinity}
    Let $\varphi$ be a formula of $\ulfo$, and let $\mathfrak{A}$ be a suitable model and $s$ a suitable assignment. If
    Eloise has a winning strategy in the game $\mathcal{G}_n(\mathfrak{A},s,\varphi)$, then she has a winning strategy in the game $\mathcal{G}_\infty(\mathfrak{A},s,\varphi)$.
\end{lemma}
\begin{proof}
    As the argument for Lemma \ref{monotonicity}, the proof is based on simulating strategies. Indeed, in $\mathcal{G}_\infty(\mathfrak{A},s,\varphi)$, by using a non-positional strategy, Eloise can keep track of the number of times the players have visited a position where the formula is a looping atom. In particular, she can thereby pretend that she is playing the game $\mathcal{G}_n(\mathfrak{A},s,\varphi)$. Thus, by Lemma \ref{positionalityall}, she also has a positional winning strategy in the game $\mathcal{G}_\infty(\mathfrak{A},s,\varphi)$.
\end{proof}

\begin{theorem}\label{thm:ulfo_validity}
    Let $\varphi$ be a formula of $\ulfo$. Now $\varphi$ is valid if and only if for some $n\in \mathbb{N}$,
    the approximant $\Phi_{\varphi}^n$ is valid.
\end{theorem} 
\begin{proof}
    The claim follows directly from Lemmas \ref{approximationswork}, \ref{monotonicityinfinity} and \ref{compactness}.
    Note that formally \ref{approximationswork} refers to $\lfo$ but is in fact independent of the the difference between $\ulfo$ and $\lfo$.
\end{proof}

\begin{corollary}\label{ulfore}
    The set of valid sentences of $\ulfo$ is recursively enumerable.
\end{corollary}

\begin{remark}\label{sameremark}
    Perhaps surprisingly, Theorems \ref{LFOVALIDITY} and \ref{thm:ulfo_validity} imply that the set of \emph{valid} sentences of $\lfo$ and $\ulfo$ coincide, i.e., $\varphi$ is valid with respect to bounded semantics iff it is valid with respect to unbounded semantics. This will be strengthened to concern first-order premise sets below when we discuss completeness.
\end{remark}

\subsection{Translating $\ulfo$ into $\uso$}

We now show how $\ulfo$ can be translated into universal second-order logic $\uso$. After that we will use this translation to determine the complexity of the validity problems of the two-variable fragments of $\lfo$ and $\ulfo$.

By $\ulfo^k$ and $\mathrm{FO}^k$ we mean, respectively, the $k$-variable fragments of $\ulfo$ and $\mathrm{FO}$. We start with the following result, which is also of independent interest.

\begin{theorem}\label{safetyexpressible}
    Let $\varphi(\overline{x}) \in \ulfo^k$ be a formula, and let $\tau$ denote the vocabulary of $\varphi(\overline{x})$. Then there exists a vocabulary $\tau' \supseteq \tau$ and a formula $\Psi(\overline{x})$ of $\fo^k$ with the following properties.
    \begin{enumerate}
        \item If $\mathfrak{A}$ is a $\tau$-model such that Eloise does not have a winning strategy in the game $\mathcal{G}_\infty(\mathfrak{A},s,\varphi)$, then $\mathfrak{A}$ can be expanded to a $\tau'$-model $\mathfrak{A}'$ so that $\mathfrak{A}',s \models \Psi(\overline{x})$.
        \item If $\mathfrak{A}'$ is a $\tau'$-model so that $\mathfrak{A},s \models \Psi(\overline{x})$, then Eloise does not have a winning strategy in the game $\mathcal{G}_\infty(\mathfrak{A} \upharpoonright \tau, s, \varphi)$.
    \end{enumerate}
    Furthermore, $\Psi(\overline{x})$ can be computed from $\varphi(\overline{x})$ in polynomial time.
\end{theorem}
\begin{proof}
    Suppose that $\{x_1,\dots ,x_k\}$ contains the set of variables occurring in $\varphi(\overline{x})$, including the bound variables. The vocabulary $\tau'$ will contain, in addition to the relation symbols in $\varphi$, the relation symbols
    $$\{N_{(\psi, \#)} \mid \psi \in \subf(\varphi),\, \# \in \{+,-\}\}$$
    where each symbol $N_{(\psi,\#)}$ is $k$-ary. The proof is similar to the proof of Lemma \ref{compactness}: the relation symbols will be used to encode positions that are safe for Abelard, by which we mean that Eloise cannot force a win from such positions.
    
    The formula $\Psi(\overline{x})$ will be a conjunction of axioms for the relations $N_{(\psi,\#)}$ which intuitively speaking  encode ``safety conditions'' for Abelard in a natural way. For instance, for every subformula $\exists x_1 \psi$ of $\varphi(\overline{x})$, we use the formulas
    \[\forall x_1 \dots x_k (N_{(\exists x_1 \psi, +)}(x_1,\dots,x_k) \to \forall x_1 N_{(\psi, +)}(x_1,\dots,x_k))\]
    and
    \[\forall x_1 \dots x_k (N_{(\exists x_1 \psi, -)}(x_1,\dots,x_k) \to \exists x_1 N_{(\psi, -)}(x_1,\dots,x_k)).\]
    As an another example, for every subformula $C_L$ with reference formula $\mathrm{Rf}(C_L)$, we use for
    both $\# \in \{+,-\}$ the formula
    $$\forall x_1 \dots x_k (N_{(C_L,\#)}(x_1,\dots,x_k) \to N_{(L\psi,\#)}(x_1,\dots,x_k)).$$
    It is clear that such axioms can be written in $\fo^k$ for every pair $(\psi,\#)$.
\end{proof}

Theorem \ref{safetyexpressible} immediately yields a polynomial time algorithm for translating formulas of $\ulfo^k$ to formulas of $\uso^k$, i.e., to
formulas $\forall X_1\dots \forall X_n\, \psi$
where the first-order part $\psi$ is a
formula of $\mathrm{FO}^k$.

\begin{corollary}\label{cor:ulfo_contained_in_uso}
    Every formula of $\varphi(\overline{x}) \in \ulfo^k$ can be translated in polynomial time to an equivalent formula $\Psi(\overline{x}) \in \uso^k$.
\end{corollary}
\begin{proof}
    Given a formula $\varphi(\overline{x}) \in \ulfo^k$,  compute the formula $\Psi(\overline{x}) \in \mathrm{FO}^k$ given by Theorem \ref{safetyexpressible}. This can clearly be computed in polynomial time from $\varphi(\overline{x})$. If $\sigma$ denotes the set of relation symbols that occur in $\Psi(\overline{x})$ but not in $\varphi(\overline{x})$,
    then $\varphi(\overline{x})$ is equivalent to the
    second-order formula
    \[\neg (\exists R)_{R\in \sigma} \Psi(\overline{x})\]
    and thereby to the formula  
    \[(\forall R)_{R\in \sigma} \neg \Psi(\overline{x})\] of $\uso^k$. 
\end{proof}

Unsurprisingly, the containment implied by Corollary \ref{cor:ulfo_contained_in_uso} turns out to be strict.

\begin{theorem}\label{theorem:ulfo_contained_in_uso}
    $\ulfo < \uso$ in relation to expressive power.
\end{theorem}
\begin{proof}
    To establish that $\uso \not\leq \ulfo$, we show that the class of finite structures over the empty vocabulary is not definable in $\ulfo$. Aiming for a contradiction, suppose that the property of finiteness is definable. Let $\varphi \in \ulfo$ be a sentence defining this class. Let $k$ be the number of variables occurring in $\varphi$. Define $\mathfrak{A}$ to be the model over the empty vocabulary and with the domain $\{0,\dots , k\}$ with precisely $k+1$ elememts. Thus
    $\mathfrak{A}\models \varphi$ 
    and hence Eloise has a winning strategy $\sigma$ in the game $\mathcal{G}_\infty(\mathfrak{A},\varphi)$. Now, let $\mathfrak{B} = \mathbb{N}$. To derive a contradiction, one can show that $\sigma$ can be used to design a winning strategy $\sigma'$ for Eloise in the game $\mathcal{G}_\infty(\mathfrak{B},\varphi)$. The idea is to simulate the ``identity information'' allowed by $\sigma$ in the game played with $\sigma'$. We next formulate this more formally.

    Let $X \subseteq \{x_1, \dots , x_k\}$, and let $S_1$ and $S_2$ be nonempty sets. Consider 
    two assignments $s_1:X\rightarrow S_1$ and $s_2:X\rightarrow S_2$. We say that the assignments are \emph{similar} if for all $x,y\in \{x_1,\dots , x_k\}$, we have $s_1(x) = s_1(y)$ iff $s_2(x) = s_2(y)$. We say that two positions in two semantic games are \emph{assignment-similar} if the assignment functions in the positions are similar. Now, it is easy to see that we can define $\sigma'$ based on $\sigma$ such that every play according to $\sigma'$ corresponds to a play
    according to $\sigma$ such that the two plays simultaneously realize assignment similar positions in every round. 
Therefore $\sigma'$ is a winning strategy in $\mathcal{G}_\infty(\mathfrak{B},\varphi)$.
\end{proof}

We can use Corollary \ref{cor:ulfo_contained_in_uso} to determine the exact complexity of the validity problem of $\ulfo^2$.

\begin{theorem}\label{thm:ulfo2_validity}
    The validity problem of $\ulfo^2$ is \textsc{coNExpTime}-complete.
\end{theorem}
\begin{proof}
    The lower bound follows directly from the corresponding lower bound for two-variable logic $\fo^2$. For the upper bound, we first note that the translation given in the proof of Theorem  \ref{theorem:ulfo_contained_in_uso} is polynomial and preserves the number of variables being used. Hence sentences of $\ulfo^2$ are translated efficiently into equivalent sentences of $\uso^2$. Now, obviously a sentence $\psi$ of $\uso$ is valid if and only if the first-order part $\chi$ of $\psi$ is valid. 
    Thus the complexity of the validity problem of $\uso^2$ is the same as $\fo^2$, namely \textsc{coNExpTime}-complete. Hence also the validity problem of $\ulfo^2$ is in \textsc{coNExpTime}.
\end{proof}

Since we have already observed that $\lfo$ and $\ulfo$ have the same set of valid sentences, we have the following further corollary, where $\lfo^2$ denotes the two-variable 
fragment of $\lfo$.

\begin{corollary}
    The validity problem of $\lfo^2$ is \textsc{coNExpTime}-complete.
\end{corollary}

We note that, by Corollary \ref{cor:ulfo_contained_in_uso}, the fact that the validity problem of $\ulfo$ is recursively enumerable follows immediately from the
validity problem of $\uso$ being,
likewise, recursively enumerable.
However, the value of having proved the validity problem of $\ulfo$ recursively enumerable via Theorem \ref{thm:ulfo_validity} lies in the fact that the theorem relates $\ulfo$-validities to structurally similar validities of first-order logic (namely, approximants of the $\ulfo$-formulas). We have already used this fact to deduce that the set of valid sentences of $\lfo$ and $\ulfo$ coincide.
This fact will be crucial also in the next section, where we design an axiomatization
which is complete for $\lfo$ as well as $\ulfo$.

\section{A complete axiomatization}

In this section we develop a
proof system for $\lfo$ and $\ulfo$
that is complete for 
valitidities and of course sound. 
Recalling that both of these logics have the 
same validities, the same system
works for both of them.
However, formally speaking, we will 
consider the case of $\ulfo$ first.
%
%
%
%
%
%

In fact, more than just completeness for 
validities will be achieved.
We shall show
that our deduction system is actually 
complete for first-order 
premise sets for both $\ulfo$ and $\lfo$,
thereby also establishing that the two 
logics coincide in their logical
consequence relations in restriction to
$\mathrm{FO}$ premise sets.

\subsection{A system of natural deduction}

The lack of classical negation must be carefully taken into account when designing a proof system for $\ulfo$ and $\lfo$. For example, the law of 
excluded middle fails to be
valid in these logics, as
demonstrated already by simple
sentences such as $LC_L \lor \neg LC_L$ and even $C_L \lor \neg C_L$.
To define a suitable deduction system, we first fix some rules.
Let $\wedge$Intro denote the rule 
$\frac{\varphi\ \ \ \psi}{\varphi\wedge\psi}$
which we may conveniently denote also by $\varphi,\psi\mapsto \varphi\wedge\psi$.
Let $\wedge$Elim$1$ and $\wedge$Elim$2$ be
the rules $\varphi\wedge\psi\mapsto\varphi$
and $\varphi\wedge\psi\mapsto\psi$, respectively, and let $\vee$Intro$1$
and $\vee$Intro$2$ denote $\varphi\mapsto\varphi\vee\psi$ and $\varphi\mapsto\psi\vee\varphi$.
Let $\vee$Elim be the rule 
\begin{center}
   \scalebox{0.9}[0.9]{$\infer[{\mbox{\tiny ($\vee$Elim)}}]{\chi}{
       \varphi\vee \psi
       &
                                   \infer*{\chi}{
               [\varphi]
       }
       &
      \infer*{\chi}{
               [\psi]
       }
}$}
\end{center}
Let $\bot$Intro be the rule $\varphi\wedge\neg\varphi\mapsto\bot$ 
and $\bot$Elim the rule $\varphi \stackrel{*}{\mapsto} \varphi[\psi_1/\bot,\dots ,\psi_k/\bot]$
which reads as follows. We begin with a formula $\varphi$ in strong negation norm (recall the definition from the preliminaries). 
We first identify a 
finite number $k$ of
occurrences of $\bot$ that are not in
the scope of any negations, and then we 
simultaneously replace them, respectively, by
arbitrarily chosen formulas $\psi_1,\dots , \psi_k$. Here the star above $\mapsto$ reminds the 
reader that the rule has the side condition that the
replaced occurrences of $\bot$ must not be in the scope of any
negations and that $\varphi$ should be in 
strong negation normal form. 
Now, the deduction system $\mathcal{S}$ we shall use is
defined as follows.

\begin{enumerate}
    \item 
    We include $\wedge$Intro,
    $\wedge$Elim$1$, $\wedge$Elim$2$,
    $\vee$Intro$1$, $\vee$Intro$2$, $\vee$Elim,
    $\bot$Elim and $\bot$Intro in $\mathcal{S}$. 
    \item
    We include enough further
    rules in $\mathcal{S}$ so 
    that together with the above rules, our 
    system becomes strongly complete for
    the language of $\mathrm{FO}$ as specified in
    this paper and remains sound for
    $\ulfo$ and $\lfo$.
    There are many ways of doing 
    this, and this is straightforward to do 
    also concerning soundness, as where
    necessary, we can add rules with side 
    conditions that limit their use to
    first-order formulas. 
    \item
    We include the six \emph{recursion operator rules}
    and six \emph{duality rules} specified below.
\end{enumerate}

The \textbf{recursion operator rules} are the following, 
with explanations on notation and side conditions given after the list.

\[
\begin{array}{c}

%
%
%
%
%
%


\scalebox{0.9}[0.9]{\infer[\updownarrow
%
%
%
%
%
%
{\mbox{\tiny(\substshift)}}]{
\varphi[L\psi\{ L\psi/C_L\}]}{
        \varphi[L\psi]
    }}

\hspace{5mm}

\scalebox{0.9}[0.9]{\infer[\updownarrow
%
%
%
%
%
%
{\mbox{\tiny(\substcopyelim)}}]{
\varphi[L\psi\{\psi/C_L\}]}{
        \varphi[L\psi]
    }}

\\ \\

    \scalebox{0.9}[0.9]{\infer[\updownarrow 
\hspace{1mm} {\mbox{\tiny($L$Dual-Intro)}}]{
\varphi[L'L\psi\{\neg C_{L'}/\neg C_L\}]}{
        \varphi[L\psi]
    }}

\hspace{5mm}

\scalebox{0.9}[0.9]{\infer[{\updownarrow
\hspace{1mm} \mbox{\tiny($L$Dummy-Intro-Elim)}}]{
\varphi[L\psi]}{
        \varphi[\psi]
}}

\end{array}
\]

\[
\begin{array}{c}

    \scalebox{0.9}[0.9]{\infer[\updownarrow 
\hspace{1mm} {\mbox{\tiny($LC_L$Rename)}}]{
\varphi[L'\psi\{\{C_{L'}/C_L\}\}\, ]}{
        \varphi[L\psi]
    }}

\hspace{5mm}

\scalebox{0.9}[0.9]{\infer[{
\hspace{1mm}\mbox{\tiny ($C_L$Free-Elim)}}]
{\varphi[\psi/C_L]}
{
       \varphi[C_L]
       }}

\end{array}
\]

\medskip

\medskip

Firstly,
both of 
the rules $\substshift$ 
and $\substcopyelim$ must 
satisfy the side-condition that the
formula on top of the horizontal line is 
regular. We shall see that the rule $LC_L$Rename will in fact 
enable modifying formulas so that they indeed
become regular.

When using the rule \substshift, we begin with a formula $\varphi[L\psi]$ 
that has a subformula occurrence $L\psi$.  
We transform that occurrence $L\psi$
to $L\psi\{ L\psi/C_L\}$, that is, to a formula obtained from $L\psi$ by
replacing \emph{some} occurrences of $C_L$
by $L\psi$ itself. Any 
subset of the occurrences of $C_L$ in $L\psi$ can be replaced; the curly brackets indicate that indeed any \emph{set} of atoms $C_L$ in $L\psi$ can be chosen to be replaced. As an
example of using
the rule, if $L\psi = L(P(x) \wedge C_L)$,
then $L\psi\{L\psi/C_L\}$
can be the 
formula $$L(P(x)
\wedge L(P(x) \wedge C_L)).$$

Notice that the
rule allows us to modify a
subformula occurrence $L\psi$ of $\varphi[L\psi]$, so we
have a \textbf{deep inference rule}, as its
use is not limited to the
modification of the main operator of $\varphi[L\psi]$. Furthermore, \substshift\ is 
\textbf{bidirectional}, which is indicated by the upright double arrow and
means that we can use the rule in the standard way as well as in the reverse
direction. More rigorously, when using a bidirectional rule, we can (1) use
the rule in the standard downward fashion, and (2) if we can 
syntactically produce a formula $\varphi_{bottom}$ from $\varphi_{top}$ in the standard way,
then we are also allowed to produce $\varphi_{top}$ from $\varphi_{bottom}$.
The rule \substcopyelim\ replaces an occurrence of $L\psi$ in $\varphi[L\psi]$ by a formula $L\psi\{\psi/C_L\}$
obtained from $L\psi$ by replacing \emph{some}
looping atoms $C_L$ by $\psi$. 
Again we are free to choose any
subset of the atoms $C_L$ in the occurrence $L\psi$ to be replaced. 
The rule $L$Dual-Intro modifies $L\psi$ so that
\emph{some}
literals $\neg C_L$ in the strict scope or $L$
are replaced by $\neg C_{L'}$ and a new corresponding label symbol $L'$ is introduced. 
We require that $L'$ is a fresh symbol not occurring anywhere in $\varphi[L\psi]$. Once again any subset of the set of literals $\neg C_L$ in
the strict scope of the $L$ in the occurrence $L\psi$ can be replaced.

The rule $L$Dummy-Intro-Elim allows us to introduce (and eliminate in
the upward direction) a dummy label symbol $L$, i.e., a symbol 
with no looping atoms $C_L$ in its strict scope. 
The rule $LC_L$Rename replaces the
occurrence $L\psi$ by 
the formula $L'\psi\{\{C_{L'}/C_L\}\}$ where we have renamed $L$ to $L'$ and replaced all atoms $C_L$ in the
strict scope of $L$ by $C_{L'}$. The double brackets indicate that indeed \emph{all} occurrences of $C_L$ in
the strict scope or the particular
instance of $L$ must be
renamed. The
symbol $L'$ can be any label symbol, as
long as (1) the formula $L\psi$ does not have any 
occurrences of $C_{L'}$ that are free in $\psi$ and (2) if $\psi$ already contains symbols $L'$, none of the occurrences of $C_L$ (that are to be replaced) are in the scope of such occurrences of $L'$. So the renaming procedure is safe. Note that while we defined $LC_L$Rename as bidirectional, this is redundant, as the top-to-bottom direction already covers what can be
achieved by the reverse direction. 
Finally, $C_L$Free-Elim (which is not a bidirectional rule) allows us to 
replace a \emph{free} looping atom $C_L$ with any formula. A free $C_L$ is an
atom not in the scope of any
occurrence of $L$.

Now, the \textbf{duality rules} are the following.

\medskip

\medskip

\[
\begin{array}{c}
\scalebox{1}[1]{\infer[\updownarrow\ {\mbox{\tiny }}]{\varphi[(\neg \psi \vee \neg \chi)]}{
       \varphi[\neg (\psi\wedge \chi)]
       }}

\hspace{3mm}

\scalebox{1}[1]{\infer[\updownarrow\ {\mbox{\tiny }}]{\varphi[(\neg \psi \wedge \neg \chi)]}{
       \varphi[\neg (\psi \vee \chi)]
       }}

\hspace{3mm}

\scalebox{1}[1]{\infer[\updownarrow\ {\mbox{\tiny}}]{\varphi[\psi]}{
       \varphi[\neg\neg\psi]
       }}\\ \\

\hspace{3mm}

\scalebox{1}[1]{\infer[\updownarrow\ {\mbox{\tiny}}]{\varphi[\exists x\, \neg\psi]}{
       \varphi[\neg \forall x\, \psi]
       }}

\hspace{3mm}

\scalebox{1}[1]{\infer[\updownarrow\ {\mbox{\tiny }}]{\varphi[\forall x\, \neg\psi]}{
       \varphi[\neg \exists x\, \psi]
       }}

\hspace{3mm}

\scalebox{1}[1]{\infer[\updownarrow\ {
     \mbox{\tiny }}]{
\varphi[L \neg \psi\{\{\neg C_L/C_L\}\}]}{
        \varphi[\neg L \psi]
    }} 
\end{array}
\]

\medskip

\medskip

In the last rule, an occurrence of $\neg L\psi$ is 
replaced by $L\neg\psi{\{\{}\neg C_L/C_L{\}\}}$, 
where $\psi{\{\{}\neg C_L/C_L{\}\}}$ is obtained from $\psi$ by
replacing \emph{every} $C_L$ (which is in 
the strict scope of the $L$ of our 
occurrence $L\psi$) by $\neg C_L$.

It is straightforward to show soundness for $\mathcal{S}$, so we
skip that proof here for the sake of brevity.

\subsection{Completeness}\label{section:axiomatization}

    Let $\varphi$ be a formula of $\ulfo$. Recall that we say that $\varphi$ is in weak negation normal form if the only negated subformulas of $\varphi$ are atomic formulas. We say that $\varphi$ is in strong negation normal form if the only negated subformulas of $\varphi$ are atomic $\fo$-formulas. 
For example $P(x)\wedge L\neg C_L$ is in weak but not strong
negation normal form because $C_L$ is not an $\fo$-atom.

%
%
%


\begin{lemma}\label{nnf}
    For any $\varphi\in\ulfo$,
    there exists a formula $\varphi^*$ in strong negation normal form such that $\varphi \vdash \varphi^*$
    and $\varphi^* \vdash \varphi$.
\end{lemma}
\begin{proof}
Firstly, we 
note that eliminating a negation 
from a \emph{free} but
negated $C_L$ can be
done by $C_L$Free-Elim and the duality 
rule for double negation. It is also easy to 
reverse this effect by $C_L$Free-Elim.

Now, to prove the claim of the lemma, we
will also use the rules \substcopyelim\ and $L\text{Dual-Intro}$ together with the duality rules.
It follows directly from the duality rules that for any formula $\beta$,
there exists a formula $\beta^*$ in weak negation normal form such that $\beta \vdash \beta^*$
and $\beta^* \vdash \beta$. Thus we assume that $\varphi$ is a formula in weak negation
normal form, and our first goal is to show how to modify $\varphi$ in a
deduction so that we get rid the possible negations in front of looping atoms $C_L$ in $\varphi$,
thereby ending up with a formula in strong negation normal form.

Now, suppose $L\psi$ is a subformula of $\varphi$ such that the following conditions hold.
\begin{enumerate}
\item
$\neg C_{L}$ occurs in $L\psi$ at 
least once in the strict scope of the main 
operator $L$ of $L\psi$.
\item
There are no label symbols in $\psi$ that are
referred to by a negated looping atom, i.e.,
there is no subformula $L_0\, \alpha$ in $\psi$
such that $\alpha$ contains the literal $\neg C_{L_0}$ in
the scope of the main operator $L_0$ of $L_0\, \alpha$.
\end{enumerate}
Then we call $L\psi$ an \emph{innermost-level switching loop formula}, or more shortly, an $\mathit{ILSL}$-formula.
To eliminate all negated
looping atoms from $\varphi$, 
we first use the rule $L\text{Dual-Intro}$ to
every $\mathit{ILSL}$-formula $L \psi$ of $\varphi$, 
thereby replacing $L \psi$ by $$L'L\psi \{\{\neg C_{L'}/\neg C_L\}\}$$
where $L'$ is a fresh label symbol. Intuitively, this just renames the literals $\neg C_L$ to $\neg C_{L'}$. 
Then, denoting $L'L\psi \{\{\neg C_{L'}/\neg C_L\}\}$
by $L'L\psi'$, we use \substcopyelim\ to
transform $L'L\psi'$ to $L'L\psi' \{L\psi' /C_{L'}\}$
where we choose to replace \emph{every} occurrence of $C_{L'}$ by $L\psi'$. As indeed every such occurrence becomes replaced, we denote $L'L\psi' \{L\psi' /C_{L'}\}$ by $L'L\psi' \{\{L\psi' /C_{L'}\}\}$.
Now, notice that each atom $C_{L'}$ occurs negated in $\psi'$, so we will therefore denote $L'L\psi' \{\{ L \psi' /  C_{L'}\}\}$ by  $L'L\psi' \{\{\neg L \psi' /\neg C_{L'}\}\}$. 
Note then that due to the formulas $\neg L \psi'$ that replace the literals $\neg C_{L'}$, the formula $L'L\psi' \{\{\neg L \psi' / \neg C_{L'}\}\}$ is no longer in weak negation normal form. Thus the next step is to push
negations to the atomic level using the duality rules, including elimination of double negations. 
Hence we obtain the formula
$L'L\psi' \{\{L \psi'_d\, / \neg C_{L'}\}\}$
where $L\psi_d'$ is the formula we ultimately obtain
from $\neg L\psi'$ via the duality rules. Now, we
claim $L\psi_d'$ contains no 
literals $\neg C_{L'}$ or $\neg C_L$.
This is due to the following observations. 
Firstly, recall from above that each atom $C_{L'}$
occurs negated in $\psi'$ while none of the atoms $C_L$ does.
Thereby, when we apply the duality rule (for label symbols) to $\neg L\psi'$, we 
obtain a formula $L\neg \psi''$ where all the atoms $C_{L'}$ and $C_{L}$ occur in $\psi''$ with a
single negation directly in front of them. Thus the
transformation of $L\neg\psi''$ to $L\psi_d'$ has the
desired effect that $L\psi_d'$ 
contains no literals $\neg C_{L'}$ or $\neg C_L$.

This way we have eliminated negated literals $\neg C_L$ from the $\mathit{ILSL}$-formulas $L\psi$ of $\varphi$. The
obtained formula
$\chi$ may still contain further 
literals $\neg C_{L''}$, but we may simply repeat the procedure described above, starting from the \emph{ILSL}-formulas of $\chi$. Altogether, the strategy is to repeat the procedure sufficiently many times until there no longer exist any \emph{ILSL}-formulas. The process ultimately terminates since after each repetition, the \emph{ILSL}-formulas of the newly obtained formula $\chi'$ will be closer to the main connective (i.e., closer to the root of the syntax tree of $\chi'$) than after the previous repetition. Indeed, the number of repetitions needed is clearly bounded above by the maximum nesting depth of label symbols in the original formula $\varphi$. We let $\varphi^*_0$ denote the final formula obtained from the procedure. Now, $\varphi_0^*$ can still contain free negated literals $C_L$, but as discussed in the beginning of the current proof, these can be eliminated by $C_L$Free-Elim and the duality rule for double negation. We let $\varphi^*$ denote the formula obtained from $\varphi_0^*$ after also the possible negations in front of free looping atoms have been eliminated.

We then discuss the converse deduction from $\varphi^*$ to $\varphi$.
First, note that all the inferences used to obtain $\varphi_0^*$
from $\varphi$ used bidirectional rules only, so we can reverse the inferences and therefore $\varphi_0^*\vdash\varphi$. Thus it suffices to show that $\varphi^*\vdash\varphi_0^*$. Now, our inference  $\varphi_0^*\vdash\varphi^*$ above simply removed the possible negations in front of \emph{free} occurrences of looping atoms. As discussed in the beginning of our proof, this last step can be reversed
simply by $C_L$Free-Elim (replacing $C_L$ by $\neg C_L$).
\end{proof}

Recall the definition of the $n$-approximants $\Phi_{\varphi}^n$ of a 
formula $\varphi$.
We are now ready to prove the
following lemma.

\begin{lemma}\label{thepreviouslemma}
Let $\varphi\in \ulfo$ be in strong
negation normal form and let $n\in \mathbb{N}$. Then $\Phi^n_\varphi \vdash \varphi$.
\end{lemma}
\begin{proof} 
Suppose $\varphi\in \ulfo$ is in strong negation normal form. 
In our argument below we can assume that $\varphi$ does not
have free looping atoms or dummy labels, and
furthermore, $\varphi$ is regular. This can be
seen as follows. 
Suppose we have proved $\Phi^n_{\varphi^*} \vdash \varphi^*$
where $\varphi^*$ is obtained from a 
regularisation $\varphi'$ of $\varphi$ by 
removing dummy labels and replacing free
looping atoms by $\bot$. Firstly, we
have $\Phi^n_{\varphi^*} 
= \Phi^n_{\varphi}$ and thus $\Phi^n_{\varphi} \vdash \varphi^*$.
Secondly, we have $\varphi^*\vdash\varphi'$ 
by $\bot$Elim and
$L$Dummy-Intro-Elim. Finally, we have $\varphi'\vdash\varphi$ by $LC_L$Rename.
Thus we can indeed make the
simplifying assumption that $\varphi$ is
regular and
has neither free looping atoms nor dummy labels. 
%

%

Now, let us define a sequence $\varphi_0,
\dots, \varphi_m$ of formulas 
where $\varphi_0 = \varphi$ 
and $\varphi_i\vdash \varphi_{i+1}$
for each $i<m$. The idea is to replace looping atoms by
corresponding reference formulas. To obtain $\varphi_{i+1}$ 
from $\varphi_i$, do the following. 
Suppose there are $\ell$ looping atom 
occurrences in $\varphi_i$. We enumerate 
these looping atoms, with the aim of replacing
them in the order of enumeration with corresponding
reference formulas, one by one.\footnote{Strictly speaking, we enumerate the paths from the root of the syntax
tree of $\varphi_i$ to the
occurrences of looping atoms, because the atoms
themselves may become renamed several times
during the next $\ell$ steps of our procedure. When we talk about the $j$th atom in the enumeration, we mean the atom whose path from the root of the syntax tree of the current formula is the same as the path of the $j$th atom in $\varphi_i$.\label{footnoteppp}}  More formally, we define a sequence of $\ell$ operations
that produce formulas $\varphi_{i,0},
\dots , \varphi_{i,\ell}$ such that 
$\varphi_i = \varphi_{i,0}$ and  
$\varphi_{i,\ell} = \varphi_{i+1}$
and we have $\varphi_{i,j}\vdash \varphi_{i,j+1}$
for each $j<\ell$. Each of the $\ell$
operations will replace a looping atom by a
corresponding reference formula in a 
way to be specified as follows.  
\begin{enumerate}
\item
Suppose we have already 
obtained $\varphi_{i,j}$
(if $j=0$, then $\varphi_{i,j} = \varphi_i$).
First use $LC_L$Rename sufficiently  
many times to obtain a
regular variant $\varphi_{i,j}^*$ of $\varphi_{i,j}$.
\item
Now, let $C_{L}$ denote the $j$th atom occurrence in
our enumeration of the looping atoms, that is, we
let $C_L$ denote the looping atom
occurrence in $\varphi_{i,j}^*$ that
corresponds to the $j$th atom in the original enumeration.
The set of looping atoms of $\varphi_{i,j}^*$ can be
different from that of $\varphi_{i,j}$, but here we are indeed replacing the atom occurrences in $\varphi_{i,j}^*$ that correspond to the original looping
atom occurrences in $\varphi_i$ (recall Footnote  
\ref{footnoteppp}). Now, we use the
rule $\substshift$ in the 
top-to-bottom direction to replace the $j$th
looping atom occurrence by the 
reference formula which that atom has in $\varphi_{i,j}^*$.
Note that regularity of $\varphi_{i,j}^*$ is required to enable us to use $\substshift$.
\end{enumerate}

This way we obtain the formula $\varphi_{i+1}$,
and thus ultimately the formula $\varphi_m$,
essentially by replacing looping atoms by 
corresponding reference formulas.

We have now established that $\varphi_0\vdash\varphi_m$ by using
$\substshift$ and $LC_L$Rename. Clearly $LC_L$Rename
has the property that if we can deduce $\beta$ from $\alpha$ by
using the rule, then we can also deduce $\alpha$
from $\beta$ via $LC_L$Rename (whence we could have
defined the rule as bidirectional). Furthermore, $\substshift$ is
bidirectional by definition. Therefore we
conclude that also $\varphi_m\vdash\varphi_0$,
that is, $\varphi_m\vdash\varphi$.

Now, intuitively $\varphi_m$ was
obtained from $\varphi = \varphi_0$ by 
repeated substitution of looping atoms by 
corresponding reference formulas. By 
making $m$ large enough, 
we obtain a formula $\Psi^* := \varphi_m$ that can
also be obtained in an alternative way
from the $n$-enfolding $\Psi_{\varphi}^n$ of
$\varphi$ by the
following two steps (note here that we do not
claim that the second one of the
steps can be reproduced by using our
deduction rules):
\begin{enumerate}
    \item
    We first rename label symbols and looping 
    atoms of $\Psi_{\varphi}^n$ in a suitable way, obtaining a
    formula $\Psi_0^*$.
    \item
    We then replace all the 
    looping atom occurrences 
    $C_{L_1},\dots , C_{L_p}$ in $\Psi_0^*$ by suitable formulas $\chi_1,\dots , \chi_p$,
    thereby ending up with $\Psi^*$.
\end{enumerate}
%
%
%
The informal key intuition is simply that we can
view $\Psi^* = \varphi_m$ as an 
extension of a renaming of 
the $n$-unfolding $\Psi_{\varphi}^n$.
We next aim to show that $\Phi_{\varphi}^n\vdash\Psi^*$
where $\Phi_{\varphi}^n$ is the $n$-approximant of $\varphi$. 
This is done as follows.

Firstly, as $\varphi$ is in strong
negation normal form, so is the $n$-unfolding $\Psi_{\varphi}^n$.
Thus the $n$-approximant $\Phi_{\varphi}^n$ is by 
definition obtained from $\Psi_{\varphi}^n$ by replacing all the
looping atoms by $\bot$ and then deleting all label symbols $L$.
Now recall from above the formula $\Psi^*_0$ obtained from $\Psi_{\varphi}^n$ by
renaming label symbols and looping atoms in $\Psi_{\varphi}^n$.
Beginning from the approximant $\Phi_{\varphi}^n$, we can 
reintroduce the corresponding label
symbols (but not looping atoms) by using the rule $L$Dummy-Intro-Elim, thus obtaining a
formula $\Phi^*$ which is otherwise as $\Psi_0^*$ but 
has atoms $\bot$ in the place of the looping
atoms $C_{L_1},\dots , C_{L_p}$ of $\Psi_0^*$.
Then, recalling the formulas $\chi_1,\dots, \chi_p$, we can
replace the atoms $\bot$ in $\Phi^*$ (corresponding to the
atoms $C_{L_1},\dots , C_{L_p}$ in $\Psi^{*}_0$) by the
formulas $\chi_1,\dots , \chi_p$,
thereby obtaining the formula $\Psi^*$. This step can be done using the rule $\bot$Elim. 
Thus we have $\Phi_{\varphi}^n\vdash\Psi^*$.

Now, as $\Phi_{\varphi}^n\vdash\Psi^*$, is
suffices to show that $\Psi^*\vdash\varphi$ to
conclude our proof. But we have already essentially
shown this. Indeed, we defined above
that $\Psi^*:=\varphi_m$ for a
suitably large $m$. Furthermore, we explicitly 
proved above that $\varphi_m\vdash\varphi$ for any $m$.
Thus we have $\Psi^*\vdash\varphi$, as required. 
\end{proof}

\begin{theorem}\label{weakcomplete}
    Let $\varphi$ be a formula of $\ulfo$ or $\lfo$. If $\varphi$ is valid, then $\vdash \varphi$.
\end{theorem}
\begin{proof}
    Suppose that $\varphi$ is a valid
    formula of $\ulfo$. Let $\varphi^*$ be the
    negation normal form variant of $\varphi$
    guaranteed to exist by Lemma \ref{nnf}. Since $\varphi$ is
    valid, so is $\varphi^*$. By Remark \ref{sameremark}
    and Theorem \ref{LFOVALIDITY}, 
    this implies that $\Phi_{\varphi^*}^n$ is
    valid for some $n\in \mathbb{N}$. Since our proof calculus is
    complete for first-order formulas, we
    have $\vdash \Phi_{\varphi^*}^n$.
    By Lemma \ref{thepreviouslemma}, we
    thus have $\vdash \varphi^*$. Hence we have $\vdash \varphi$ by Lemma \ref{nnf}, concluding the
    case for $\ulfo$. The 
    case for $\lfo$ now follows from Remark \ref{sameremark}. 
\end{proof}

We the
show that if $\varphi$ is a formula of $\ulfo$ or $\lfo$ and $\Sigma$ is a set of $\fo$-formulas,
then we have $\Sigma \models \varphi$ iff $\Sigma \vdash \varphi$, i.e., we have 
completeness with respect to $\mathrm{FO}$ premise sets.

\begin{lemma}\label{compactnessforfoassumptionslfo}
    Let $\varphi$ be a formula of $\lfo$ or $\ulfo$ and let $\Sigma$ be a set of $\fo$-formulas. Suppose that $\Sigma \models \varphi$. Then there exists a finite $\Sigma_0 \subseteq \Sigma$ so that $\Sigma_0 \models \varphi$.
\end{lemma}
\begin{proof}
    We first consider the case where $\varphi$ is in \lfo.
    Suppose that for every finite $\Sigma_0 \subseteq \Sigma$, there exists some $\mathfrak{A}$ and $s$ so that $\mathfrak{A},s \models \Sigma_0$ but $\mathfrak{A},s\not\models \varphi$. By Theorem \ref{approximationsequivalent}, for all 
    such $\mathfrak{A}$ and $s$, we have
    $\mathfrak{A} \models \neg \Phi_\varphi^n$
    for every $n\in \mathbb{N}$. Thus, by compactness of $\fo$, we can deduce that
    $\Sigma \cup \{\neg \Phi_\varphi^n \mid n\in \mathbb{N}\}$
    is satisfiable. This implies, by
    Theorem \ref{approximationsequivalent}, that $\Sigma \not\models \varphi$, contradicting the assumption that $\Sigma\models\varphi$.


We then consider the case where $\varphi$ is in \ulfo.
    Suppose that for every finite $\Sigma_0 \subseteq \Sigma$, there exist $\mathfrak{A}$ and $s$ so that $\mathfrak{A},s \models \Sigma_0$ but $\mathfrak{A},s \not\models \varphi$. Every such model has an
    expansion $\mathfrak{A}'$ to a larger vocabulary so that
    $\mathfrak{A}',s \models \Psi$
    where $\Psi$ is the formula promised by Theorem \ref{safetyexpressible}.
    Thus, by compactness of $\fo$, we see that
$\Sigma \cup \{\Psi\}$
    is satisfiable. But if some $\mathfrak{B}$ and $s$ satisfy this theory, then---due to the properties of the formula $\Psi$---Eloise does not have a winning strategy in the game $\mathcal{G}_{\infty}(\mathfrak{B}',s,\varphi)$ where $\mathfrak{B}'$ is the restriction of $\mathfrak{B}$ to the vocabulary of $\varphi$.
    Thus $\Sigma\not\models\varphi$. 
\end{proof}

\begin{lemma}\label{lemma:fo-deduction-theorem}
    Let $\varphi$ be a formula of $\lfo$ or $\ulfo$. Let $\Sigma$ be a finite set of $\fo$-formulas. Now $\Sigma \vdash \varphi$ if and only if $\vdash (\neg \bigwedge \Sigma) \lor \varphi$.
\end{lemma}
\begin{proof}
    Suppose $\Sigma \vdash \varphi$. Thus $\bigwedge\Sigma\vdash\varphi$ by
    $\wedge$Elim$1$ and $\wedge$Elim$2$. Therefore $\bigwedge\Sigma\vdash (\neg\bigwedge\Sigma) \vee \varphi$ by $\vee$Intro$2$. 
    On the other hand, we have $\neg\bigwedge\Sigma\vdash (\neg\bigwedge\Sigma)
    \vee \varphi$ by $\vee$Intro$1$. As our system is complete for 
    first-order logic, we have $\vdash(\neg\bigwedge\Sigma)
    \vee \bigwedge\Sigma$. Combining this with the above established facts that $\bigwedge\Sigma\vdash (\neg\bigwedge\Sigma) \vee \varphi$ and $\neg\bigwedge\Sigma\vdash (\neg\bigwedge\Sigma)
    \vee \varphi$, we
    conclude by $\vee$Elim that $\vdash (\neg\bigwedge\Sigma)\vee\varphi$.

    Suppose $\vdash (\neg\bigwedge\Sigma) \vee \varphi$.
    We need to show that $\Sigma\vdash\varphi$. 
    We have $\Sigma\vdash \bigwedge\Sigma$ 
    by $\wedge$Intro.
    Thus $\Sigma\cup\{\neg\bigwedge\Sigma\}\vdash\varphi$ due to $\bot$Intro and $\bot$Elim. As
    also $\Sigma \cup\{\varphi\}\vdash\varphi$, we
    have $\Sigma \cup \{(\neg\bigwedge\Sigma)\vee\varphi\}
    \vdash \varphi$ by $\vee$Elim. Thus, as we
    have assumed that $\vdash (\neg\bigwedge\Sigma) \vee \varphi$,
    we have $\Sigma\vdash \varphi$.
\end{proof}

\begin{theorem}\label{focompletness}
    Let $\varphi$ be a formula of $\lfo$ or $\ulfo$. Let $\Sigma$ be a set of $\fo$-formulas. Now $\Sigma \models \varphi$
    iff $\Sigma \vdash \varphi$.
\end{theorem}
\begin{proof}
The right-to-left direction follows from soundness. For the other direction, suppose that $\Sigma \models \varphi$, where $\Sigma$ is a set of $\fo$-sentences and $\varphi$ is a formula of either $\lfo$ or $\ulfo$. By Lemma \ref{compactnessforfoassumptionslfo}, there exists a finite set $\Sigma_0 \subseteq \Sigma$ so that $\Sigma_0 \models \varphi$, i.e., $(\neg \bigwedge \Sigma_0) \vee \varphi$ is valid. By Theorem \ref{weakcomplete}, we
have that $\vdash (\neg \bigwedge \Sigma_0) \vee \varphi$. Using Lemma \ref{lemma:fo-deduction-theorem}, we have  $\Sigma_0 \vdash \varphi$, and
hence $\Sigma \vdash \varphi$. 
\end{proof}

Now note that we can express in $\ulfo$ that a linear order is well-founded, and this can clearly be used to define $(\mathbb{N},+,\times,0,1)$ up to isomorphism with a single sentence of $\ulfo$. Therefore we
cannot upgrade the above theorem so that $\Sigma$ is a set of $\ulfo$-formulas, as the equivalence $\Sigma\models\varphi\Leftrightarrow\Sigma\vdash\varphi$ would imply that true arithmetical $\fo$-sentences would form a recursively enumerable set.

A similar limitation
holds for $\lfo$. To see this, 
let $\varphi_{<}$ be a formula defining that $<$ is a strict discrete linear order with end points, and $S$ is the corresponding successor order. Now, consider the following formula (where we use ``min'' and ``max'' as
constants that indicate the
end points; it is clear that the 
constants can be eliminated in order to 
keep the vocabulary entirely relational): 
\[\varphi_{<} \land (\mathrm{min} = \mathrm{max} \lor \exists x (S(\mathrm{min},x) \land L  (x = \mathrm{max} \lor \exists y (S(x,y) \land \exists x (x = y \land C_L)))).\]
The formula, let us denote it by $\psi$, essentially states that there is a finite path from $\mathrm{min}$ to $\mathrm{max}$, which implies that the domain of the underlying model must be finite (without imposing any finite upper bound on its size). Now, let $R$ denote a fresh binary relation. It is easy to see that, for every $\varphi$ in the first-order language over the vocabulary $\{R\}$, we have that $\psi \models \varphi$ iff $\varphi$ is valid over the class of finite $\{R\}$-models. It follows quite directly from Trakhtenbrot's theorem that validity over finite $\{R\}$-models is $\Pi_1^0$-complete, which implies that the consequences of $\psi$ form a set that is not recursively enumerable.

\section{Some model theory of $\lfo$ and $\ulfo$}

The purpose of this section is to present preliminary results on the model theory of $\lfo$ and $\ulfo$. Given that both of these logics are non-compact, it is unlikely that they have as rich model theory as, say, $\fo$. However, both of these logics can be seen as fragments of infinitary logics, which in turn do admit nice model theories (even though they are also often non-compact). This gives us hope that one could also develop nice model theories for $\lfo$ and $\ulfo$.

\subsection{Löwenheim-Skolem}

We say that a logic $\mathcal{L}$ has \textbf{countable downwards Löwenheim-Skolem property}, if every sentence $\varphi$ of $\mathcal{L}$ has the following property: if $\mathfrak{A} \models \varphi$, then $\mathfrak{A}$ has a countable substructure $\mathfrak{B}$ which is also a model of $\varphi$. As advertised in the introduction, both $\lfo$ and $\ulfo$ have the countable downwards Löwenheim-Skolem propety.

We start by establishing this for $\lfo$, for which it follows almost directly from the fact that $\fo$ has the countable downwards Löwenheim-Skolem property.

\begin{theorem}
    Let $\varphi$ be a sentence of $\lfo$ and suppose that $\mathfrak{A} \models \varphi$. Then there exists a countable substructure $\mathfrak{B}$ of $\mathfrak{A}$ such that $\mathfrak{B} \models \varphi$.
\end{theorem}
\begin{proof}
    Suppose that $\mathfrak{A} \models \varphi$. Theorem \ref{approximationsequivalent} implies that $\mathfrak{A} \models \bigvee_{n \in \mathbb{N}} \Phi_\varphi^n$. Hence $\mathfrak{A} \models \Phi_\varphi^n$, for some $n \in \mathbb{N}$. Since $\Phi_\varphi^n$ is a sentence of $\fo$, we know that there exists a countable substructure $\mathfrak{B}$ of $\mathfrak{A}$ such that $\mathfrak{B} \models \Phi_\varphi^n$. Using theorem \ref{approximationsequivalent} again, we conclude that $\mathfrak{B} \models \varphi$.
\end{proof}

In the case of $\ulfo$ it turns out that we can adapt the standard proof that $\fo$ has the countable downwards Löwenheim-Skolem property.

\begin{theorem}\label{lowenheimskolem}
    Let $\varphi$ be a sentence of $\ulfo$ and suppose that $\mathfrak{A} \models \varphi$. Then there exists a countable substructure $\mathfrak{B}$ of $\mathfrak{A}$ such that $\mathfrak{B} \models \varphi$.
\end{theorem}
\begin{proof}
To simplify notation, we may assume that $\varphi$ 
    has only quantifiers $\exists$ 
    by writing $\neg\exists \neg$ instead of $\forall$ in the usual way.
    Suppose that $\varphi$ has a model $\mathfrak{A}$ so that Eloise has a winning strategy $\sigma$ in the game $\mathcal{G}_\infty(\mathfrak{A},\varphi)$.
    We may assume the strategy is positional by Lemma \ref{reachabilitygametheorem}. We want to construct a countable model $\mathfrak{B}$ so that Eloise has a winning strategy also in the game $\mathcal{G}_\infty(\mathfrak{B},\varphi)$. 
    
    Pick an arbitrary $b \in A$. We define a sequence of sets $(B_n)_{n\in \mathbb{N}}$ inductively as follows.
    \begin{enumerate}
        \item $B_0 = \{b\}$.
        \item $B_{n+1} = B_n \cup \{d \mid \sigma((\exists x \psi, s, +)) = d \text{, where }\ran(s) \subseteq B_n\}$.
    \end{enumerate}
    Let $\mathfrak{B}$ be the substructure of $\mathfrak{A}$ induced by the set $\bigcup_{n\in \mathbb{N}} B_n$. $\mathfrak{B}$ is clearly countable.
    
    It is easy to see that $\sigma$, or more precisely its restriction to the set of positions occurring in $\mathcal{G}_\infty(\mathfrak{B},\varphi)$, is also a winning strategy for Eloise also in $\mathcal{G}_\infty(\mathfrak{B},\varphi)$. Indeed, as long as Eloise follows it in $\mathcal{G}_\infty(\mathfrak{B},\varphi)$, which is possible by the definition of $\mathfrak{B}$, Eloise will eventually reach---after a finite number of rounds---a winning position.
\end{proof}

\subsection{Craig interpolation property}

A logic $\mathcal{L}$ has the \textbf{Craig interpolation property}, if the following holds for every two sentences $\varphi$ and $\psi$ of $\mathcal{L}$: if $\varphi \models \psi$, then there exists a third sentence $\theta \in \mathcal{L}$ called an \textbf{interpolant}, such that $\varphi \models \theta \models \psi$ and $\theta$ contains only those relation symbols that occur in both of the sentences $\varphi$ and $\psi$. We will next establish that neither $\lfo$ nor $\ulfo$ has the Craig interpolation property. These results should be contrasted with the fact that several infinitary logics, such as $\mathcal{L}_{\omega_1 \omega}$, do enjoy the Craig interpolation property.

We start by establishing that the class of finite structures of even size is not definable in neither $\lfo$ nor in $\ulfo$.

\begin{proposition}\label{prop:finite-even-not-definable}
   For every sentence $\varphi$ of either $\lfo$ or $\ulfo$ there exists a finite structure $\mathfrak{A}$ of even size and a finite structure $\mathfrak{B}$ of odd size such that
   \[\mathfrak{A} \models \varphi \Rightarrow \mathfrak{B} \models \varphi.\]
\end{proposition}
\begin{proof}
    Since the expressive power of $\lfo$ and $\ulfo$ coincides over finite models, it suffices to consider the case of $\lfo$. Let $\varphi$ be an arbitrary sentence of $\lfo$. Suppose that $\varphi$ contains $k$ distinct variables. By Theorem \ref{approximationsequivalent} we know that $\varphi$ is equivalent with the sentence
    \[\Phi := \bigvee_{n \in \mathbb{N}} \Phi_\varphi^n\]
    of $\mathcal{L}_{\omega_1 \omega}^k$. Consider now the models $\mathfrak{A}$ and $\mathfrak{B}$, where both are models over the empty vocabulary with domains $\{0,\dots,2k\}$ and $\{0,\dots,2k+1\}$ respectively.  Now, it is easy to show using pebble games that these structures can not be distinguished via a sentence of $\mathcal{L}_{\omega_1 \omega}^k$. In particular, if $\mathfrak{A} \models \Phi$, then $\mathfrak{B} \models \Phi$. This in turn entails that if $\mathfrak{A} \models \varphi$, then $\mathfrak{B} \models \varphi$.
\end{proof}

We will next establish the failure results. Our counterexample is inspired by the standard counterexample which shows that over finite models $\fo$ does not enjoy Craig interpolation property. It is not surprising that a similar example could be made to work also in our case, since both $\lfo$ and $\ulfo$ can projectively define the class of finite structures (which $\fo$ can not do, since it has compactness).

\begin{theorem}
    Neither $\lfo$ nor $\ulfo$ has the Craig interpolation property.
\end{theorem}
\begin{proof}
    Recall the sentence $\psi$ that we introduced in the end of Section \ref{section:axiomatization}:
    \[\varphi_{<} \land (\mathrm{min} = \mathrm{max} \lor \exists x (S(\mathrm{min},x) \land L  (x = \mathrm{max} \lor \exists y (S(x,y) \land \exists x (x = y \land C_L)))).\]
    The main properties of this sentence were the following.
    \begin{enumerate}
        \item If $\mathfrak{A} \models \psi$, then $A$ is finite.
        \item If $\mathfrak{A}$ is a finite structure over a vocabulary which is disjoint from that of $\psi$, then it has an extension $\hat{\mathfrak{A}}$ such that $\hat{\mathfrak{A}} \models \psi$.
    \end{enumerate}
    Both of these properties hold regardless of whether we are using bounded or unbounded semantics.
    
    Now consider the sentences
    \[\chi_1 := \theta_1 \land \forall x \exists y (x \neq y \ \land \ E_1(x,y) \ \land \ \forall z (E_1(x,z) \to (x = z \ \lor \ y = z)))\]
    and
    \[\chi_2 := \theta_2 \ \land \ \exists x (\forall y (x \neq y \ \to \ \neg E_2(x,y))\]
    \[\land \ \forall y (y \neq x \to \exists z (y \neq z \ \land \ E_1(y,z) \ \land \ \forall w (E_1(y,w) \to (y = w \ \lor \ z = w)))))\]
    where $\theta_i$, for $i \in \{1,2\}$, expresses that $E_i$ is an equivalence relation. Note that the common vocabulary of $\chi_1$ and $\chi_2$ is the empty vocabulary. It is easy to see that $\chi_1$ expresses that $E_1$ is an equivalence relation where each equivalence class contains precisely two elements, while $\chi_2$ is expressing that $E_2$ is an equivalence relation where there exists one equivalence class with one element while every other equivalence class has precisely two elements.
    
    Clearly $\psi \land \chi_1 \models \neg \chi_2$, since $\mathfrak{A} \models \psi \land \chi_1$ entails that $|A|$ is finite and even, while $\mathfrak{A} \models \chi_2$ would entail that $|A|$ is either infinite or even. We now claim that there exists no interpolant between $\psi \land \chi_1$ and $\chi_2$ either in $\lfo$ or in $\ulfo$. 
    
    Aiming for a contradiction, suppose that $\theta$ is a sentence of either $\lfo$ or $\ulfo$ over the empty vocabulary which is an interpolant between $\psi \land \chi_1$ and $\chi_2$. Let $\mathfrak{A}$ and $\mathfrak{B}$ be the structures promised by Proposition \ref{prop:finite-even-not-definable}. $\mathfrak{A}$ clearly has an extension $\hat{\mathfrak{A}}$ such that $\hat{\mathfrak{A}} \models \psi \land \chi_1$. Since $\theta$ was interpolant, we have that $\hat{\mathfrak{A}} \models \theta$, which implies that $\mathfrak{A} \models \theta$, since $\theta$ was a sentence over the empty vocabulary. Thus $\mathfrak{B} \models \theta$. Now, $\mathfrak{B}$ clearly has an extension $\hat{\mathfrak{B}}$ for which $\hat{\mathfrak{B}} \models \chi_2$, since $B$ had odd size. But now also $\hat{\mathfrak{B}} \models \theta$, which is a contradiction, since $\theta \models \neg \chi_2$.
\end{proof}

\subsection{Sentences that are determined everywhere}

We have seen several examples of sentences of $\lfo$ and $\ulfo$ which can define properties of classes of models which are not definable by any sentence of $\fo$. In each case one can make the observation that the relevant sentence has a model in which it is non-determined, i.e., neither player has a winning strategy. For instance, the $\ulfo$ sentence which defined the class of well-founded linear orders is non-determined in any model which contains an infinite descending sequence.

This raises the following question: if a sentence of $\lfo$ or $\ulfo$ defines a class of models which is not definable by any $\fo$-sentence, must it be non-determined in some model? The answer turns out to be positive in both cases; if a sentence of either $\lfo$ or $\ulfo$ is determined everywhere, then it is in fact (strongly) equivalent to one of its approximants.

\begin{theorem}\label{thm:determined_everywhere}
    Suppose that $\varphi$ is a sentence of either $\lfo$ or $\ulfo$, which is determined everywhere. Then $\varphi$ is equivalent to a sentence of $\fo$ and more specifically it is equivalent with its $n$th approximant, for some $n$.
\end{theorem}
\begin{proof}
    We will first consider the case where $\varphi$ is a sentence of $\lfo$. Since $\varphi$ is determined everywhere, the sentence $\varphi \lor \neg \varphi$ is valid, which implies --- together with Theorem \ref{LFOVALIDITY} --- that $\Phi_{\varphi \lor \neg \varphi}^n$ is valid. Note that $\Phi_{\varphi \lor \neg \varphi}^n$ is the same as $\Phi_\varphi^n \lor \Phi_{\neg \varphi}^n$. Now, we claim that $\Phi_\varphi^n$ is in fact equivalent with $\varphi$. Recall that Theorem \ref{approximationsequivalent} implies that $\Phi_\psi^n \models \psi$, for every sentence $\psi$. Thus in particular $\Phi_\varphi^n \models \varphi$. Concerning the other direction $\varphi \models \Phi_\varphi^n$, we note that since $\Phi_{\varphi}^n \lor \Phi_{\neg \varphi}^n$ is valid, $\neg \Phi_\varphi^n \models \Phi_{\neg \varphi}^n \models \neg \varphi$. Hence $\varphi$ is equivalent with $\Phi_\varphi^n$.
    
    The case where $\varphi$ is a sentence of $\ulfo$ can be handled analogously. The main differences are that instead of Theorem \ref{LFOVALIDITY} we use Theorem \ref{thm:ulfo_validity}, and the fact that $\Phi_\psi^n \models \psi$ holds for every sentence $\psi$ of $\ulfo$ follows from lemmas \ref{monotonicityinfinity} and \ref{approximationswork}, instead of Theorem \ref{approximationsequivalent}.
\end{proof}

We now make two remarks concerning the question of to what extend our result can be made effective. We start by determining the exact complexity of the problem of determining whether a given sentence of $\lfo$ or $\ulfo$ is determined everywhere.

\begin{proposition}\label{prop:complexity-determined-everywhere}
   The problem of determining whether a given sentence $\varphi$ of either $\lfo$ or $\ulfo$ is determined everywhere is $\Sigma_1^0$-complete.
\end{proposition}
\begin{proof}
    Let $\mathcal{L} \in \{\lfo, \ulfo\}$. We have seen in the previous sections that the set of valid sentences of $\mathcal{L}$ is a recursively enumerable set. This fact already implies that the set of sentences of $\mathcal{L}$ which are determined everywhere is a recursively enumerable set; an effective procedure can simply go through the list of valid sentences of $\mathcal{L}$, and print the sentence $\varphi$ whenever it encounters the sentence $\varphi \lor \neg \varphi$.
    
    For the lower bound we will reduce the validity problem of $\fo$ to the problem of determining whether a sentence of $\mathcal{L}$ is determined everywhere. Let $\varphi \in \fo$ be a sentence. We claim that $\varphi$ is valid iff the sentence
    \[\psi_\varphi := \varphi \lor C_L\]
    is determined everywhere. First, if $\varphi$ is valid, then $\psi_\varphi$ is determined everywhere, because it is a valid sentence. Conversely, if $\psi_\varphi$ is determined everywhere, then $\varphi$ must be valid, since $C_L$ is non-determined in every model.
\end{proof}

An immediate corollary of the above result is that the problem of determining whether a given sentence of $\lfo$ or $\ulfo$ is strongly equivalent to a sentence of $\fo$ is also $\Sigma_1^0$-complete.

\begin{corollary}
    The problem of determining whether a given sentence $\varphi$ of either $\lfo$ or $\ulfo$ is strongly equivalent to a sentence of $\fo$ is $\Sigma_1^0$-complete.
\end{corollary}
\begin{proof}
    Let $\mathcal{L} \in \{\lfo, \ulfo\}$. We have already established that a sentence of $\mathcal{L}$ is strongly equivalent with a sentence of $\fo$ if and only if it is determined everywhere.\footnote{Note that a sentence of $\mathcal{L}$ might be \emph{weakly} equivalent with a sentence of $\fo$ and yet be undetermined in some models. A concrete example of such a sentence is $\exists x \exists y R(x,y) \lor C_L$.} Thus the claim follows from Proposition \ref{prop:complexity-determined-everywhere}.
\end{proof}

To complement these results, we note that if $\varphi$ is determined everywhere, then we can effectively recover a sentence of $\fo$ which is equivalent with $\varphi$. This follows from the observation that if $\varphi$ is equivalent to some $\fo$-sentence $\psi$, then for some $n\in \mathbb{N}$ we have that $(\varphi \leftrightarrow \Phi_\varphi^n)$ is a valid sentence of $\mathcal{L}$, since $\varphi$ was determined everywhere. This allows us, together with the fact that the set of valid sentences of $\mathcal{L}$ is recursively enumerable, recover $\Phi_\varphi^n$ effectively.

Finally we will give an example which demonstrates that Theorem \ref{thm:determined_everywhere} fails if we restrict our attention to the class of finite models. 

\begin{example}
    Consider the following sentence
    \[\varphi := \varphi_{<} \land \exists x \exists y (x = \min \ \land \  y = \max \ \land \ (S(x,y) \ \lor \]
    \[L\exists z \exists w (S(x,z) \ \land \ S(w,y) \ \land \ (S(z,w) \ \lor \exists x \exists y (x = z \ \land \ y = w \ \land \ C_L))))),\]
    where $\varphi_{<}$ expresses that $<$ is a strict linear ordering of the domain, $\min$ and $\max$ are distinct elements that correspond to the smallest and the largest elements of $<$ and $S$ is the successor relation induced by $<$. It is easy to see that, regardless of whether we are using the bounded or unbounded game-theoretical semantics, $\varphi$ defines the class of linear orders of even size. It is well-known that this class is not $\fo$-definable, and hence $\varphi$ is not equivalent to any sentence of $\fo$. 
    
    Next we will show that $\varphi$ is determined everywhere. Suppose that $\mathfrak{A}$ is a suitable model for $\varphi$. If $\mathfrak{A} \not\models \varphi_<$, then Abelard clearly has a winning strategy. Suppose then that $\mathfrak{A} \models \varphi_<$, but $|A|$ is not an even number. To see that Abelard has a winning strategy also in this case, note that $\varphi$ describes a game where Eloise needs to move two pebbles along the successor relation induced by $\varphi_<^\mathfrak{A}$, the initial position of these pebbles being the smallest and the largest elements of $\varphi_<^\mathfrak{A}$. Now, if $|A|$ is not an even number, Eloise will eventually reach a position where there is only a single element between the two pebbles, which is a position that is outside her winning region, because she needs to maintain the condition that the first pebble is always placed on an element which is strictly smaller than the element to which the second pebble is placed. We note that if we are using bounded semantics, then Abelard additionally needs to make sure that the initial clock value for $L$ is large enough.
\end{example}

\section{Definability over natural numbers}

The purpose of this section is to characterise relations over natural numbers that are definable in $\ulfo$ and in $\lfo$ over the standard structure $\mathbb{N}$ of natural numbers. We start by formally defining the classes $\Pi_1^1$ and $\Sigma_{\omega+1}^0$ starting with the former. A relation $X\subseteq \mathbb{N}^k$ is called $\Pi_1^1$ if there exists a formula
\[\forall X_1 \dots \forall X_n \psi(x_1,\dots,x_k)\]
of $\uso$ such that for every $(m_1,\dots,m_k) \in \mathbb{N}^k$ we have that $(m_1,\dots,m_k) \in X$ if and only if $\mathbb{N} \models \psi(m_1,\dots,m_k)$.

To define the class $\Sigma_{\omega+1}^0$, we start by fixing some (reasonable) Gödel numbering $\Godelnum{}$ for the formulas of $\fo$-arithmetic. What we mean by reasonable should become clear in our proofs. Now, consider the set
\[T := \{ \Godelnum{\varphi} \mid \varphi \in \fo \text{ is a sentence in prenex normal form and } \mathbb{N} \models \varphi\}.\]
A relation $X\subseteq \mathbb{N}^k$ is called $\Sigma_{\omega + 1}^0$, if there exists a $T$-computable relation $R$ such that
\[(x_1,\dots,x_k) \in X \Leftrightarrow \exists y_1 \dots \exists y_\ell R(x_1,\dots,x_k,y_1,\dots,y_\ell).\]
Here by $T$-computable, we mean that the relation can be computed by a Turing machine that has a distinct (oracle) tape where the characteristic function $\chi_T$ of $T$ is written down. In other words, the machine has access to a tape that contains the infinite sequence $\chi_T(0),\chi_T(1),\dots$.

We note that an alternative --- and perhaps a more standard --- way of defining the class $\Sigma_{\omega+1}^0$ would be to use the set $\varnothing^{(\omega)}$ instead of $T$ (for a formal definition of $\varnothing^{(\omega)}$, see \cite[p. 257]{Hartley87}). Since the two sets are recursively isomorphic, meaning that there exists a computable bijection $f:\mathbb{N} \to \mathbb{N}$ so that $n \in T$ if and only if $f(n) \in \varnothing^{(\omega)}$, the two definitions of $\Sigma_{\omega + 1}^0$ coincide, see \cite[p. 318]{Hartley87}. However, in our case it is technically more convenient to work with the set $T$.

\subsection{$\ulfo$-definable relations}

By Theorem \ref{theorem:ulfo_contained_in_uso} we know that $\ulfo$ is contained in $\uso$ over general structures, and hence every relation over $\mathbb{N}$ that is definable in $\ulfo$ is also $\Pi_1^1$. To prove the converse direction, we will modify the proof of Kleene's theorem as presented in the book \cite{KOZENBOOK}.

\begin{lemma}\label{lemma:pi11-contained-ulfo}
    Every $\Pi_1^1$-relation over $\mathbb{N}$ is definable in $\ulfo$.
\end{lemma}
\begin{proof}
    We start by observing that, over $\mathbb{N}$, it is routine to rewrite an arbitrary $\uso$ formula $\varphi(\overline{x})$ as a formula of the form
    \[\forall f \exists y \varphi(\overline{x},y),\]
    where $f$ is an unary function and $\varphi(\overline{x},y)$ is quantifier-free. Thus we need to show that for each such formula there exists --- over $\mathbb{N}$ --- an equivalent $\ulfo$ formula. For simplicity, we will restrict our attention to the case where $\varphi$ contains a single free variable, i.e., we consider formulas where the quantifier-free part is of the form $\varphi(x,y)$.
    
    The basic idea is now as follows. Given any function $f:\mathbb{N} \to \mathbb{N}$ and natural numbers $m,m' \in \mathbb{N}$, we can determine whether $\varphi(m,m')$ holds in $\mathbb{N}$ by considering $f\upharpoonright n$ -- the restriction of $f$ to $\{0,\dots,n-1\}$ -- for some sufficiently large $n$. Thus the evaluation of $\forall f \exists y \varphi(x,y)$ can be formulated as the following game: Abelard picks natural numbers $f(0),f(1),f(2),\dots$ until Eloise chooses to stop the game and evaluate the formula $\exists y \varphi(m,y)$ with the restriction of $f$ induced by the natural numbers that were chosen by Abelard.
    
    Now we construct an $\ulfo$-formula $\theta(x)$ which essentially describes the above game. To do this, we will first need to fix some effective method of encoding tuples of natural numbers as a single natural number. A standard choice of encoding is 
    \[(n_1,\dots,n_k) \mapsto \prod_{i=1}^k p_i^{n_i + 1},\]
    where $p_i$ denotes the $i$th prime number. Now consider the function $g:\mathbb{N}^2 \to \mathbb{N}$ defined by $(n_1,n_2) \mapsto n_1 \times p^{n_2 + 1}$, where $p$ is the smallest prime number which does not divide $n_1$ (in the case where $n_1 = 1$, we simply set $p$ to be $2$). Since $g$ is effectively computable, there exists a formula $\psi(x,y,z)$ of $\fo$-arithmetic which defines it.
    
    The formula $\theta(x)$ can now be defined as the formula
    \[\exists n \exists f [n = 1 \land f = 1 \land L(\exists y \varphi^*(n,f,x,y) \lor \]
    \[\forall z \exists n' \exists f' (n' = n + 1 \land \psi(f,z,f') \land \exists n \exists f (n = n' \land f = f' \land C_L)))]\]
    where $\varphi^*(n,f,x,y)$ is a formula which will be specified later, but, roughly speaking, it is false if the values chosen by Abelard are not sufficient to determine whether $\varphi(x,y)$ holds and true if $\varphi(x,y)$ is true when evaluated under the mapping determine by the values chosen by Abelard. Based on the above discussion, it should be clear that $\theta(x)$ is the desired formula.
    
    To define the formula $\varphi^*(n,f,x,y)$ we proceed as follows. First, consider the mapping $h:\mathbb{N}^2 \to \mathbb{N}$ defined by
    \[\bigg(\prod_{i=1}^k p_i^{n_i}, j\bigg) \mapsto 
    \begin{cases*}
        n_j - 1 & if  $1\leq j\leq k$ and $n_j\geq 1$  \\
        0 & otherwise
    \end{cases*}\]
    Again, this mapping is clearly computable, and hence there exists a formula $\chi(x,y,z)$ of $\fo$-arithmetic which defines it.
    
    Consider now an arbitrary atomic formula $\alpha(x,y)$ of $\varphi(x,y)$. Our goal is to write, for every such formula $\alpha(x,y)$, a formula $\psi_\alpha(n,f,x,y)$ which essentially verifies that the values chosen by Abelard -- which are encoded in the number $f$ -- are indeed enough to determine whether $\alpha(x,y)$ holds. Having such formulas at hand, we will replace each such atomic formula $\alpha(x,y)$ of $\varphi(n,f,x,y)$ with the corresponding formula $\psi_{\alpha}(n,f,x,y)$. The resulting formula will then the desired formula $\varphi^*(n,f,x,y)$.
    
    We start with a more concrete example. Consider an atomic formula $\alpha(x,y)$ that contains only the terms $x,y,f(x),f(f(x))$. Consider then the formula
    \[(x < n \to \exists (\chi(f,x,z) \land (z < n \to \exists w (\chi(f,z,w) \land \alpha(x,y)[z/f(x),w/f(f(x))]))).\]
    The formula starts by verifying that $f$ encodes the value of $f(x)$, which it does as long as $x < n$, since then $f(x) + 1$ is the exponent of the $x$th prime number that divides $f$. Having verified this, the formula then stores $f(x)$ into the variable $z$. Next, the formula verifies that $f$ also encodes the value $f(f(x)) = f(z)$, and then stores $f(f(x))$ into the variable $w$. Finally, the formula verifies that $\alpha(x,y)[z/f(x),w/f(f(x))]$ holds. Clearly this formula could now be used as the formula $\psi_\alpha(n,f,x,y)$.
    
    In general the atomic formula $\alpha(x,y)$ contains terms from the set \[\{x,f(x),\dots,f^t(x),y,f(y),\dots,f^s(y)\},\]
    which will make the resulting formulas $\psi_\alpha(n,f,x,y)$ more complicated. On the other hand, it is easy to see that the technique that was used in the above example generalizes also to handle this more general case.
\end{proof}

The following is immediate.

\begin{theorem}\label{thm:ulfo_pi11_over_N}
    Over $\mathbb{N}$, $\ulfo$-definable relations and $\Pi_1^1$-relations coincide.
\end{theorem}
\begin{proof}
    Combine Theorem \ref{theorem:ulfo_contained_in_uso} with Lemma \ref{lemma:pi11-contained-ulfo}.
\end{proof}

\subsection{$\lfo$-definable relations}

We will next establish that the class of $\lfo$-definable relations and the $\Sigma_{\omega+1}^0$-relations coincide. We will start by establishing that every $\lfo$-definable relation is $\Sigma_{\omega+1}^0$.

\begin{lemma}\label{lemma:lfo-contained-sigmaomegaplusone}
    Every $\lfo$-definable relation over $\mathbb{N}$ is $\Sigma_{\omega+1}^0$.
\end{lemma}
\begin{proof}
    Suppose that $X\subseteq \mathbb{N}^k$ is defined by the formula $\varphi(\overline{x}) \in \lfo$. By Theorem \ref{approximationsequivalent}, we have that $\overline{m} \in X$ iff for some $n\in \mathbb{N}$ it is the case that $\mathbb{N} \models \Phi_n^\varphi(\overline{m})$. Let $R(\overline{x},y)$ be a relation which is satisfied by those pairs $(\overline{m},n) \in \mathbb{N}^{k+1}$ for which the $n$th approximant of $\varphi(\overline{x})$ is satisfied by $\overline{m}$ in $\mathbb{N}$. Since $\Phi_y^\varphi(\overline{x})$ is, for every $y \in \mathbb{N}$, a formula of $\fo$-arithmetic which can be computed from $\varphi(\overline{x})$ when given $y$, $R(\overline{x},y)$ is clearly $T$-computable. Hence $\exists y R(\overline{x},y)$ is a $\Sigma_{\omega+1}^0$-definition of $X$.
\end{proof}

To prove the converse direction, we will first show that in a certain technical sense the set $T$ is itself $\lfo$-definable.

\begin{lemma}\label{lemma:empty-jump-expressible}
    There exists a formula $\theta(x) \in \lfo$ so that for every sentence of $\fo$-arithmetic in prenex normal form with quantifier-depth at most $d$ we have that the following two conditions hold.
    \begin{enumerate}
        \item Eloise has a winning strategy in the game $\mathcal{G}_d(\mathbb{N},r,\theta(x))$ if and only if $\mathbb{N} \models \varphi$.
        \item Eloise has a winning strategy in the game
        $\mathcal{G}_d(\mathbb{N},r,\neg \theta(x))$ if and only if $\mathbb{N} \not\models \varphi$.
    \end{enumerate}
    Here $r$ is an assignment for which $r(x) = \Godelnum{\varphi}$.
\end{lemma}
\begin{proof}
    Consider a sentence $\varphi$ of $\fo$-arithmetic which is in prenex normal form. There clearly exists a computable function which when given as input $\Godelnum{\varphi}$ computes the length of the prefix of $\varphi$. Furthermore, we can effectively determine from $\Godelnum{\varphi}$ whether the $i$th quantifier in the prefix of $\varphi$ is universal. Let $\psi_{\mathrm{prefix}}(x,y)$ denote a formula defining the first function and let $\psi_{\mathrm{universal}}(x,i)$ denote a formula which is true iff the $i$th quantifier in the sentence encoded by $x$ is universal. Let $\psi(x,y,z)$ denote a formula defining the function $g:\mathbb{N} \times \mathbb{N} \to \mathbb{N}$ that we used in the proof of Lemma \ref{lemma:pi11-contained-ulfo}. Consider now the following formula $\theta(x)$ of $\lfo$
    \[\exists d \exists n \exists m [\psi_{\mathrm{prefix}}(x,d) \land n = d \land m = 1 \land L((n = 0 \to \psi(x,m)) \land (n > 0 \to \]
    \[(\psi_{\mathrm{universal}}(x,d-(n-1)) \to \exists n' \forall y \exists m' (n' = n - 1 \land \psi(m,y,m')\]
    \[\land \exists n\exists m (n = n' \land m = m' \land C_L)))\]
    \[\land (\neg \psi_{\mathrm{universal}}(x,d-(n-1)) \to \exists n' \exists y \exists m' (n' = n - 1 \land \psi(m,y,m')\]
    \[\land \exists n \exists m (n = n' \land m = m' \land C_L)))))],\]
    where $\psi(x,m)$ is true iff the quantifier-free part of the sentence encoded by $x$ is true under the assignment encoded by $m$. 
    
    Now, roughly speaking, $\theta(x)$ describes a game where Abelard and Eloise choose interpretations for variables that are being quantified in the formula encoded by $x$. Eloise chooses values for the existentially quantified variables, while Abelard chooses values for the universally quantified variables. After the players have chosen $d$ values, where $d$ is the length of the quantifier prefix of the input sentence, Eloise looses if the resulting assignment $m$ does not satisfy the quantifier-free part of the formula, and otherwise Abelard looses. (Note that there are no plays where neither Eloise nor Abelard wins.) It is straightforward to verify that $\theta(x)$ satisfies both conditions stated in the lemma.
\end{proof}

\begin{lemma}\label{lemma:sigmaomegaplusone-contained-lfo}
    Every $\Sigma_{\omega+1}^0$-relation is $\lfo$-definable.
\end{lemma}
\begin{proof}
    Suppose that $X\subseteq \mathbb{N}^k$ is $\Sigma_{\omega+1}^0$. Thus there exists a $T$-computable relation $R(\overline{x},y_1,\dots,y_\ell)$ so that
    \begin{equation}\label{eq:def_of_sigma_omega+1}
        \overline{n}\in X \Leftrightarrow \exists y_1 \dots \exists y_\ell R(\overline{n},y_1,\dots,y_\ell).
    \end{equation}
    Suppose that $M(\overline{x},y_1,\dots,y_\ell)$ is a Turing machine which computes $R$ when it has oracle access to the set $T$. Now (\ref{eq:def_of_sigma_omega+1}) can be rewritten as
    \begin{equation}\label{eq:redef_of_sigma_omega+1}
        \overline{n} \in X \Leftrightarrow \exists y_1 \dots \exists y_\ell \exists t \text{ ``$M$ halts on input $(\overline{n},y_1,\dots,y_\ell)$ after $t$ steps"}
    \end{equation}
    Observe that if $M$ halts after $t$ steps, then it could have only accessed the first $t$-bits on the oracle tape. This simple observation will play a crucial role in our proof.
    
    A number $m$ is called $t$-\textbf{good}, if its prime factorization is of the form
    \[\prod_{i=1}^t p_i^{e_i},\]
    where $e_i \in \{1,2\}$, for every $1\leq i\leq t$. In other words, $m$ is $t$-good if it encodes a binary sequence of length $t$. Now the following relation is clearly computable:
    \[\text{``$m$ is $t$-good, for some $t\geq 1$, and $M$ halts after at most $t$-steps, if $(\overline{x},y_1,\dots,y_\ell)$ is}\]
    \[\text{on the input tape and the binary sequence encoded by $m$ is on the oracle tape."}\]
    Let $\varphi_M(m,\overline{x},y_1,\dots,y_\ell)$ denote a formula of $\fo$-arithmetic which defines this relation. Consider now the following formula of $\lfo$
    \[\exists y_1 \dots \exists y_\ell \exists m [\varphi_M(m,\overline{x},y_1,\dots,y_\ell)\]
    \[\land \exists t (\psi_1(m,t) \land \forall i (1\leq i\leq t \to ((\neg \psi_2(m,i) \lor \theta(i)) \land (\psi_2(m,i) \lor \neg \theta(i)))))]\]
    where $\theta$ is the formula given by Lemma \ref{lemma:empty-jump-expressible}, while the formulas $\psi_1$ and $\psi_2$ have the following meaning: $\psi_1(m,t)$ is true iff $m$ is $t$-good; and $\psi_2(m,i)$ is true iff the $i$th bit in the bit sequence encoded by $m$ is one. It is straightforward to verify that this formula defines the relation $X$.
\end{proof}

The following is immediate.

\begin{theorem}\label{thm:sigmaomeplusone_equals_lfo}
    Over $\mathbb{N}$, $\lfo$-definable relations and $\Sigma_{\omega+1}^0$-relations coincide.
\end{theorem}
\begin{proof}
    Combine Lemma \ref{lemma:sigmaomegaplusone-contained-lfo} with Lemma \ref{lemma:lfo-contained-sigmaomegaplusone}.
\end{proof}

We conclude this section with the observation that it seems likely that one can generalize the proof of Corollary \ref{thm:sigmaomeplusone_equals_lfo} to show that stronger variants of $\lfo$ are able to capture $\Sigma_{\alpha}^0$-relations for every computable ordinal $\alpha$. (We note that in the case $\alpha = \omega$ we have by definition that $\Sigma_\omega^0 = \Delta_\omega^0$, i.e., $\Sigma_\omega^0$ is the class of arithmetical relations which is already captured by $\fo$.) Here by stronger variants we mean variants where the initial value chosen by Eloise is not a natural number, but rather some computable ordinal. For instance, $\Sigma_{\omega+2}^0$-relations should be captured by the variant of $\lfo$ where the players can force the initial clock value to be $\omega$.

\bibliographystyle{plainurl}
\bibliography{kirjasto}

\end{document}